\documentclass[11pt,a4paper]{amsart}
\usepackage[headings]{fullpage}
\usepackage{amssymb,amscd,hyperref}
\usepackage{amsthm}
\usepackage[alphabetic, nobysame]{amsrefs}
\usepackage[all,cmtip]{xy}
\usepackage{color}
\usepackage{xcolor}
\usepackage{tikz-cd}
\numberwithin{equation}{section}
\newtheorem{thm}{Theorem}[section]
\newtheorem{lem}[thm]{Lemma}
\newtheorem{prop}[thm]{Proposition}
\newtheorem{cor}[thm]{Corollary}

\theoremstyle{definition}
\newtheorem{defn}[thm]{Definition}
\newtheorem{ex}[thm]{Example}

\theoremstyle{remark}
\newtheorem{rem}[thm]{Remark}
\newtheorem*{acknowledgements}{Acknowledgements}

\def\thr{\mathsf{THR}}

\def\bigbox{\raisebox{-1mm}{\scalebox{1.5}{$\square$}}}

\def\fix{\mathrm{fix}}

\newcommand{\und}[1]{{\underline{#1}}}
\def\norm{\mathsf{norm}}
\def\tr{\mathsf{tr}}
\def\res{\mathsf{res}}

\def\ra{\rightarrow}
\def\leq{\leqslant}
\def\geq{\geqslant}

\def\cC{\mathcal{C}}
\def\cL{\mathcal{L}}

\def\ie{\emph{i.e.}}

\def\id{\mathrm{id}}

\begin{document}
\title{Real Hochschild homology as an equivariant Loday
  construction}

\author{Ayelet Lindenstrauss}
\address{Mathematics Department, Indiana University, 831 East Third Street,
  Bloomington, IN 47405, USA}
\email{alindens@iu.edu}

\author{Birgit Richter}
\address{Fachbereich Mathematik der Universit\"at Hamburg,
  Bundesstra{\ss}e 55, 20146 Hamburg,  Germany}
\email{birgit.richter@uni-hamburg.de}

\author{Foling Zou}
\address{Academy of Mathematics and Systems Science, Chinese Academy of Sciences,
 55 East Zhongguancun Road, Beijing, 100190, China}
\email{zoufoling@amss.ac.cn}

\date{\today}

\begin{abstract}
  Equivariant Loday constructions are a means for providing geometric
  interpretations of equivariant homology theories. They are usually
  constructed for a simplicial 
  $G$-set and a $G$-Tambara functor. We study situations where --
  depending on the isotropy subgroups occurring in the simplicial $G$-set -- one can work with $H$-Tambara functors
  for a suitable subgroup $H$ of $G$. We apply this to give an interpretation
  of Real Hochschild homology of discrete $E_\sigma$-rings as
  equivariant Loday constructions  
  where we consider $2m$-gons with a geometrically defined action of the
  dihedral groups $D_{2m}$ for all $m \geq 1$.  The action of
  symmetric groups on $1$-skeleta of permutohedra also gives examples
  with isotropy groups $C_2$. 
  \end{abstract}
\maketitle

\section{Introduction}

For Green and Tambara functors for a finite group $G$ one can consider
several homology theories. The simplest one is equivariant Hochschild
homology where one just replaces the tensor product in ordinary
Hochschild homology with the box product of Mackey functors; see for
instance \cite[\S 4]{mqs} for some of its properties. 
A variant of this is the twisted cyclic nerve of
Blumberg-Gerhardt-Hill-Lawson \cite{bghl} for $C_n$-Green functors
where the last face map is twisted by the group action. 

Real topological Hochschild homology has an algebraic counterpart for
$C_2$-Tambara functors or more generally for discrete $E_\sigma$-rings
such as fixed point Tambara functors of associative rings with a
$C_2$-action via anti-involution. This
homology can be modelled as the homotopy groups of a two-sided bar
construction $B(\und{R}, N_e^{C_2}i^{C_2}_e \und{R}, \und{R})$, where
$  N_e^{C_2}i^{C_2}_e \und{R}$ denotes the norm restriction of
$\und{R}$.

All the homology theories above can
be modelled by equivariant Loday constructions as defined in
\cite{lrz-gloday}: The input for a $G$-equivariant Loday construction
is a simplicial 
finite $G$-set $X$ and a $G$-Tambara functor 
$\und{T}$. We then form a simplicial $G$-Tambara functor
$\cL^G_X(\und{T})$ whose $n$-simplices are
\[ \cL^G_X(\und{T})_n = X_n \otimes \und{T}. \]
As $X_n$ is a finite $G$-set, it is of the form $X_n =
\bigsqcup_{i=1}^m G/H_i$ for some isotropy subgroups $H_i < G$. Work of
Mazur \cite{mazur} (see also \cite[\S 5]{hm}) and Hoyer \cite{hoyer} ensures that the assignment 
\[ X_n \otimes \und{T} := N_{H_1}^G i_{H_1}^G\und{T} \Box \ldots \Box
  N_{H_m}^G i_{H_m}^G\und{T} \]
is well-defined and natural in $X_n$ and therefore gives rise to a simplicial
$G$-Tambara functor. Here, $i_{H_j}^G\und{T}$ denotes the restriction
of the $G$-Tambara functor $\und{T}$ to an $H_j$-Tambara functor and $N_{H_j}^G
\colon H_j\text{-Tamb} \ra G\text{-Tamb}$ is the left adjoint functor
to restriction, called the norm.

In suitable situations we show how to define $G$-Loday constructions
for $H$-Tambara functors for $H < G$. Our motivating example is Real
Hochschild homology in the sense of \cite[Definition 1.4]{akgh}: For a
discrete $E_\sigma$-ring $\und{R}$ the \emph{Real $D_{2m}$-Hochschild
homology of $\und{R}$} is the graded $D_{2m}$-Mackey functor
\[ \pi_*B(N_{D_2}^{D_{2m}}\und{R}, N_e^{D_{2m}}i_e^{D_2}\und{R},
  N_{\zeta D_2\zeta^{-1}}^{D_{2m}}c_\zeta(\und{R})). \]
Here, $D_{2m}$ denotes the dihedral group with $2m$
elements. In particular, $D_2$-Tambara functors are examples of discrete
$E_\sigma$-rings. Thus the above homology theory starts with something
with $D_2$-symmetry and produces an output with $D_{2m}$-symmetry.

For a ring spectrum $A$ with anti-involution Chloe Lewis \cite[Theorem 3.9]{lewis} constructs a B\"okstedt type spectral sequence whose $E^2$-page consists of Real Hochschild homology groups of the $E$-homology of $A$  which converges to the $E$-homology of the Real topological Hochschild homology of $A$: 
\[ E^2_{*,\star} = \underline{\mathsf{HR}}_*^{E_\star, D_{2m}}(\underline{(i_{D_2}^{D_{2m}}E)}_\star(A)) \Rightarrow \underline{E}_\star(i_{D_{2m}}^{O(2)}(\thr(A))).\]
Here, $E$ is a genuine commutative $D_{2m}$ ring spectrum, and there are some flatness and freeness requirements to be satisfied. As $\thr(A)$ is the target of a trace map from the Real K-theory of $A$, $KR(A)$, calculating these group is of importance.

We give a topological interpretation of Real $D_{2m}$-Hochschild
homology for all $m \geq 1$ by letting $D_{2m}$ act on (the
$1$-skeleton of a simplicial model of) a regular $2m$-gon $P_{2m}$.

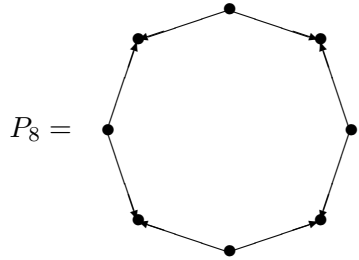
\vspace{2cm}

\hspace{4cm}
\begin{picture}(10,5)
\setlength{\unitlength}{0.4cm}
\put(0,0){$\bullet$}
\put(1,3){$\bullet$}
\put(4,4){$\bullet$}

\put(7,3){$\bullet$}

\put(8,0){$\bullet$}
\put(7,-3){$\bullet$}

\put(4,-4){$\bullet$}
\put(1,-3){$\bullet$}
\put(0.2,0.2){\vector(1,3){1}}
\put(0.2,0.2){\vector(1,-3){1}}
\put(8.2,0.2){\vector(-1,3){1}}
\put(8.2,0.2){\vector(-1,-3){1}}
\put(4.2,4.2){\vector(3,-1){3}}
\put(4.2,4.2){\vector(-3,-1){3}}
\put(4.2,-3.8){\vector(3,1){3}}
\put(4.2,-3.8){\vector(-3,1){3}}
\put(-3,0){$P_8 = $}
\end{picture}

\vspace{1.7cm}

In the above example the isotropy subgroups occurring in $P_8$ are the
trivial group $e$, $\langle s\rangle$, $\langle r^2s = 
sr^2\rangle$, $\langle
rs\rangle$, and $\langle r^3s = 
  sr\rangle$. Here we denote by $s \in D_{2m}$ a generator of order two and $r$
has order $m$. We can model the corresponding simplicial sets $P_{2m}$
by just
considering the isotropy subgroups $e$, $\langle s\rangle$ and $\langle
rs\rangle$ for all $m$.  A priori we
would need a $D_{2m}$-Tambara functor for a $D_{2m}$-Loday
construction but by carefully adjusting our Loday construction, we can
extend its definition so that it accepts a $D_2$-Tambara functor as input. We
denote
the corresponding Loday construction by $\cL_{P_{2m}}^{D_{2m};D_2}(-)$ to denote
the ambient symmetry group -- in this case $D_{2m}$ -- but also the isotropy
group that determines which type of Tambara functors to use (in this case
$D_2$). This yields: 

\bigskip \noindent
\textbf{Theorem} \ref{thm:RealHH}. Assume that $\und{R}$ is a
$D_2$-Tambara functor and let $\varphi 
  \colon D_{2m} \cong D_{2m}$ be the automorphism of $D_{2m}$ determined
  by $\varphi(rs)=s$ and $\varphi(r)=r$. Then
  \[ \mathcal{L}_{P_{2m}}^{D_{2m};D_2}(\und{R}) \cong B(N^{D_{2m}}_{D_2}
    \und{R}, N^{D_{2m}}_e i_e \und{R}, N^{D_{2m}}_{D'_2} \varphi^*(\und{R})),
  \]
 where $D_2'=\varphi^{-1}(D_2)$,  and its homotopy groups are isomorphic to the Real $D_{2m}$-Hochschild
  homology
  $\und{HR}_*^{D_{2m}}(\und{R})$ of $\und{R}$ viewed as a discrete
  $E_\sigma$-ring with values in graded $D_{2m}$-Mackey functors.

\medskip  
We extend the above result and show that the Loday constructions
$\cL_{P_{2m}}^{D_{2m};D_2}(\und{R})$ can actually be defined for discrete
$E_\sigma$-rings (see Proposition \ref{prop:esigma}). We also show that in the
setting of equivariant stable homotopy theory $i^{O(2)}_{D_{2m}}\thr(A)$ can also
be interpreted as a Loday construction for all ring spectra $A$ with
anti-involution (Proposition \ref{prop:thr}) 
\[ i^{O(2)}_{D_{2m}}\thr(A) \simeq \cL_{P_{2m}}^{D_{2m};D_2}(A) \]
and that this yields an isomorphism on $\und{\pi}_0$ if $A$ is connective:
\[ \und{\pi}_0^{D_{2m}}(i^{O(2)}_{D_{2m}}\thr(A)) \cong \und{\pi}_0^{D_{2m}}\cL_{P_{2m}}^{D_{2m};D_2}(A) \cong \cL_{P_{2m}}^{D_{2m};D_2}(\und{\pi}_0^{D_2} A). \]

We first show in section \ref{sec:auto} how to transform an
$H'$-Tambara functor into an 
$H$-Tambara functor if $\varphi$ is an isomorphisms between $H$ and $H'$. We
will later apply this in situations where $\varphi$ is the restriction of an
automorphism of $G$.  Note that for even
$m$ the subgroups $\langle s\rangle$ and $\langle rs\rangle$ are
\emph{not} conjugate in $D_{2m}$, so we specifically need a setting
where $\varphi$ is not necessarily an inner automorphism. 

Another subtle point concerns group actions on norms. In
\cite[Remark 2.5]{lrz-gloday} we used the diagonal Weyl group action
of $H$ in $G$, $W_G(H)$, on 
$N_{H}^Gi_{H}^G\und{R}$ which is a 
combination of the action that permutes tensor factors and a 
coordinatewise action.  For instance if
$G = C_2$ and $H = e$, then $N_e^{C_2}i_e^{C_2}\und{R}(C_2/e) = \und{R}(C_2/e)
\otimes \und{R}(C_2/e)$ and with the diagonal action the generator  $\tau \in C_2$ acts by
sending $a \otimes b$ to
$\tau(b) \otimes \tau(a)$. For $G$-Loday constructions of $G$-Tambara
functors this is one possible \emph{choice} and this choice for
instance yields an identification of the twisted cyclic nerve with an
equivariant Loday construction \cite[Proposition 7.3]{lrz-gloday}. We
could also choose the Weyl group action on $N_{H}^Gi_{H}^G\und{R}$ where
$W_G(H)$ just acts by permuting the tensor factors. We call the latter the
flip action. Beware that this only
works if one adjusts the counit maps $N_{H}^Gi_{H}^G\und{R} \ra
\und{R}$: For the diagonal action, the counit map at the free level is
just the multiplication, but for the flip action the augmentation is
more complicated. For $H=e$, one sends an element
$\bigotimes_{g \in G} t_g$ in $\und{T}(G/e)^{\otimes G}$ to $\prod_{g
  \in G} g(t_g)$. See section \ref{sec:norms} for a detailed discussion. 

In section \ref{sec:isotropy} we show how to define a $G$-equivariant Loday
construction with respect to a finite $G$-simplicial set $X$ and an
$H$-Tambara functor $\und{R}$ if the isotropy subgroups occurring in a
model of $X$ are  $e$
and $H$ and the conjugates of $H$ for some subgroup $H <
G$. If $H$ is normal in $G$, then a simplicial model of a Cayley graph
for $G/H$ viewed as a $G$-simplicial set is a natural example. The family of
permutohedra also gives rise to examples: 
We consider the  Loday constructions for symmetric groups where we let
$\Sigma_n$ act on the $1$-skeleton of the $n$-permutohedron. The isotropy
groups in these
cases are generated by transpositions, so they are all conjugate to
each other. 

We also extend this result to some cases where  the
isotropy consists of subgroups $e,H,H'$ and their conjugates. Here we assume that there is an automorphism of $G$,
$\varphi$, with $H' = \varphi(H)$. This gives rise to the example of Real
Hochschild homology and we prove the identification of the corresponding Loday construction with Real Hochschild homology for discrete $E_\sigma$-rings in
section \ref{sec:RealHH}.

\begin{acknowledgements}

The first named author was supported by a grant from the Simons
Foundation (\#359565, Ayelet Lindenstrauss). The second named author thanks
Churchill College Cambridge for its hospitality.    
The second and third named author would like to thank the Isaac Newton Institute
for Mathematical Sciences, Cambridge, for support and hospitality
during the programme Equivariant homotopy theory in context, where
work on this paper was undertaken. This work was supported by EPSRC
grant EP/Z000580/1. 

The first named author would like to thank Ben Spitz and Michael Larsen for
useful conversations. The authors thank John Rognes for spotting a superfluous
assumption and Mike Hill for pointing out that the two possible $G$-actions on norm restrictions $N_e^Gi_e^G\und{R}$ are isomorphic. 
\end{acknowledgements}

\section{Change of groups for Tambara
  functors} \label{sec:auto}
In the context of Tambara functors the effect of conjugation maps is
well understood. We need a slightly more general property. The following result
is probably well-known, but we couldn't find a reference in the literature:

\begin{lem} \label{lem:varphir} 
Assume that $H$ and $H'$ are finite groups and that $\varphi \colon H \ra H'$ is an isomorphism. Let $\und{R}$ be a $H'$-Tambara functor. Then the assignment 
\[ (\varphi^*\und{R})(H/K) := \und{R}(\varphi H/\varphi K) \]
gives rise to an  $H$-Tambara functor $\varphi^*\und{R}$.  
\end{lem}

\begin{proof}
We first define the structure maps of $\varphi^*\und{R}$. 

\begin{enumerate}
\item 
For any $h \in H$ we have to define $c_h \colon (\varphi^*\und{R})(H/K) \ra (\varphi^*\und{R})(H/hKh^{-1})$, \ie, we need $c_h \colon \und{R}(\varphi H/\varphi K) \ra \und{R}(\varphi H/ \varphi(hKh^{-1}))$ and we define the conjugation map as 
\[ c_h = c_h(\varphi^*\und{R}) := c_{\varphi(h)}(\und{R}). \]  

\item 
Similarly, the restriction $\res_K^L \colon (\varphi^*\und{R})(H/L) \ra (\varphi^*\und{R})(H/K)$ has to be a map from $\und{R}(\varphi H/\varphi L)$ to $\und{R}(\varphi H/\varphi K)$ and we define it as 
\[ \res_K^L = \res_K^L(\varphi^*\und{R}) := \res_{\varphi K}^{\varphi L}(\und{R}). \] 

\item 
The transfer $\tr_K^L \colon (\varphi^*\und{R})(H/K) \ra (\varphi^*\und{R})(H/L)$ is defined as $\tr_K^L = \tr_K^L(\varphi^*\und{R}) := \tr_{\varphi K}^{\varphi L}(\und{R})$. 
\item 
Last, but not least, the norm $\norm_K^L \colon  (\varphi^*\und{R})(H/K) \ra (\varphi^*\und{R})(H/L)$ is given by $\norm_K^L = \norm_K^L(\varphi^*\und{R}) := \norm_{\varphi K}^{\varphi L}(\und{R})$.  

\end{enumerate}

With these definitions it is straightforward to see that $\varphi^*\und{R}$ satisfies the Mackey and Green functor axioms. It is also clear that the norm maps are maps of multiplicative monoids, that they interact nicely with conjugation, that they satisfy transitivity for subgroups $L' < L < K$ and that they obey a double coset formula. 

For the exponential formula we have to say how we transform an arbitrary
$H$-set $X$ into a $\varphi H$-set $\varphi X$. If we express $X$ via the
orbit decomposition as $X = H/K_1 \sqcup \ldots \sqcup H/K_\ell$, we know that
disjoint unions map to products under Mackey functors. Therefore we have to
show that different choices of representatives for orbits result in coherent
isomorphisms of Tambara functors.
So assume that $X = H/H_x \cong Hx$ is an orbit. If
we choose $y \in Hx$, then $y = hx$ for one $h \in H$ and the stabilizer of $y$, $H_y$, is equal to $hH_xh^{-1}$. This yields
\[ \und{R}(\varphi H/\varphi H_x) = (\varphi^*\und{R})(H/H_x)
  = (\varphi^*\und{R})(H/h^{-1}H_yh) = \und{R}\big(\varphi H/(\varphi(h))^{-1}\varphi
  H_y \varphi(h)\big)\]
and the latter receives the isomorphism
\[ c_{\varphi(h)^{-1}} \colon (\varphi^*\und{R})(H/H_y) = \und{R}(\varphi H/\varphi H_y) \ra \und{R}\big(\varphi H/(\varphi(h))^{-1}\varphi
  H_y \varphi(h)\big). \] 
So we get nothing but the usual ambiguity when defining Tambara functors on orbits which can be healed by conjugation isomorphisms \cite[Corollary 2.4.6]{hoyer}.

If $f \colon X \ra Y$ is an $H$-equivariant map of finite $H$-sets, then we define $\varphi_\bullet(f) \colon \varphi X \ra \varphi Y$ as $\varphi_\bullet(f) = \varphi \circ f \circ
\varphi^{-1}$. Then $\varphi_{\bullet}(f)$ is an
$\varphi(H)$-equivariant map.

\bigskip 
Assume that for given $f$ and $g$ 
\[\xymatrix{  X \ar[d]_f & A \ar[l]_g & X \times_Y \prod_f A
    \ar[d]^{pr_2}\ar[l]_(0.6)\varepsilon \\ Y & & \prod_f A  \ar[ll]_q } \]
is an exponential diagram with $\prod_f A = \{(y,s), y \in Y, s \colon f^{-1}(y) \ra A \text{ with } g(s(x)) = x \text{ for all } x \in f^{-1}(y)\}$, $\varepsilon(x,y,s) = s(x)$ and $q(y,s) = y$. We claim that 
\[\xymatrix{  \varphi X \ar[d]_{\varphi_\bullet(f)} & \varphi A \ar[l]_{\varphi_\bullet(g)} & & \varphi X \times_{\varphi Y} \prod_{\varphi_\bullet(f)} \varphi A
    \ar[d]^{\varphi_\bullet(pr_2) = pr_2}\ar[ll]_(0.6){\varphi_\bullet(\varepsilon)} \\ \varphi Y & & & \prod_{\varphi_\bullet(f)}  \varphi A  \ar[lll]_{\varphi_\bullet(q)} } \]
is an  exponential diagram: A  section $s \colon f^{-1}(y) \ra A$ is sent to the section 
\[ \varphi_\bullet(s)
\colon (\varphi_\bullet(f))^{-1}(\varphi(y)) = (\varphi \circ f \circ
\varphi^{-1})^{-1}(\varphi(y)) \ra \varphi A\]
and hence $\varphi(\prod_f A) = \prod_{\varphi_\bullet(f)} \varphi A$. 
Applying $\varphi$ to finite $H$-sets and $\varphi_\bullet$ to $H$-equivariant maps preserves pullbacks. This proves the claim. As $\und{R}$ is an $\varphi H$-Tambara functor, we get for the structure maps for $\und{R}$
\[ \norm_{\varphi_\bullet(f)} \circ \tr_{\varphi_\bullet (g)} = \tr_{\varphi_\bullet (q)} \circ \norm_{pr_2} \circ \res_{\varphi_\bullet(\varepsilon)} \]
and this gives
\[ \norm_{f} \circ \tr_{g} = \tr_{q} \circ \norm_{pr_2} \circ \res_{\varepsilon}\]
for $\varphi^*\und{R}$. 
\end{proof}

  

\begin{rem}
  Pulling back Tambara functors behaves  contravariantly, as usual:
  If $\varphi$ and
  $\psi$ are two isomorphisms with $\varphi(H) = H'$ and
  $\psi(H')=H''$, then for any $H''$-Tambara functor $\und{R}$ we
  obtain an $H$-Tambara functor 
  \[ (\psi \circ \varphi)^*(\und{R}) = \varphi^*(\psi^*(\und{R})). \]

  Later, we will use the lemma above if  $\varphi \colon G \ra G$ is an
  automorphism of $G$ whose restriction to a subgroup $H$ gives rise to an
  isomorphism between $H$ and $\varphi(H)$. 
  \end{rem}

  If $X$ is a genuine $H'$-spectrum and $\varphi \colon H \ra H'$ is an isomorphism, then we can pull $X$ back to $\varphi^*X$ where the latter is now a genuine
  $H$-spectrum. 
  The external $H$-action
  on $\varphi^*X$ is then defined via
  \[ \xymatrix@1{H_+ \wedge \varphi^*X \ar[rr]^{\varphi_+\wedge \id} &&
      H'_+ \wedge X \ar[r] & X, }\] 
  where the second map is the given $H'$-action on $X$. In particular, for any subgroup $K < H$ we get $(\varphi^*X)^K = X^{\varphi(K)}$. For every genuine $H'$-spectrum $X$, $\und{\pi}_0^{H'}X$ is an $H'$-Mackey functor. 

  \begin{lem} \label{lem:pi0} 
    If $X$ is a connective genuine $H'$-spectrum and if $\varphi$ is an isomorphism of finite groups $\varphi \colon H \ra H'$, then
    the pullback construction is compatible with taking $\und{\pi}_0$: 
    \[ \und{\pi}_0^{H}(\varphi^*X) \cong \varphi^*(\und{\pi}_0^{H'}(X)). \]
  \end{lem}
\begin{proof}
  By definition, the value of the Mackey functor $\und{\pi}_0(X)$ on an orbit $H'/K'$ is $\pi_0(X^{K'})$. Hence
  \[ \varphi^*(\und{\pi}_0^{H'}(X))(H/K) = \und{\pi}_0^{H'}(X)(\varphi(H)/\varphi(K)) = \pi_0(X^{\varphi(K)}) = \pi_0((\varphi^*X)
    ^K) = \und{\pi}_0^{H}(\varphi^*X)(H/K). \] 
\end{proof}
Pulling back $H'$-spectra to $H$-spectra is symmetric monoidal, so we get a corresponding statement for genuine $H'$-commutative ring spectra and $H'$-Tambara functors.

\section{Group actions on norms} \label{sec:norms}
If we use the spectrum definition of the norm of an $H$-Tambara functor
$\und{R}$ using its Eilenberg Mac Lane spectrum $H\und{R}$ (while suppressing
change-of-universe functors) like in \cite[A.52, A.57]{hhr} as 
\begin{equation} \label{eq:normspectrum} 
  N_H^G\und{R} := \und{\pi}_0^G\big(\bigwedge_{G/H} H\und{R}\big), \end{equation}
then by definition of the functor $\und{\pi}_0$ 
\[ N_H^G\und{R}(G/K) = \und{\pi}_0^G\big(\bigwedge_{G/H}
  H\und{R}\big)(G/K) =  \pi_0\big((\bigwedge_{G/H} H\und{R})^K \big)
\]

The $G$-action on the spectrum $\bigwedge_{G/H}
H\und{R}$ can be described on the level of $\pi_0$
when we evaluate at the free level $G/e$: 
\[ N_H^G\und{R}(G/e) =
\und{\pi}_0^G\big(\bigwedge_{G/H}
H\und{R}\big)(G/e) = \pi_0\big(\bigwedge_{G/H}
H\und{R}\big) \cong \bigotimes_{G/H} \und{R}(H/e). \] 
Here permutation actions on smash factors are visible in the
corresponding permutation action of tensor factors and an 
action by $H$ on the Eilenberg-Mac Lane spectrum of $\und{R}$
corresponds to an $H$-action on $\und{R}(H/e)$. In the following we
will therefore describe the $G$-action on norms at the free level. 

For a $G$-Tambara functor $\und{T}$ we can consider its restriction to the
trivial subgroup, $i_e^G\und{T}$. A priori, this is just an $e$-Tambara functor,
thus a  commutative ring, with 
\[ i_e^G\und{T}(e/e) = \und{T}(G \times_e e/e) \cong \und{T}(G/e). \]
But the $G$-Tambara structure of $\und{T}$ ensures that $\und{T}(G/e)$ carries
the structure of a commutative $G$-ring. This is important for the counit of
the norm-restriction adjunction
\[ \varepsilon \colon N_e^Gi_e^G\und{T} \ra \und{T}. \]
This counit map has to be a morphism of $G$-Tambara functors, in particular,
at the free level $G/e$ we have to have a $G$-map
\[ \varepsilon_e \colon N_e^Gi_e^G\und{T}(G/e) = i_e^G\und{T}(e/e)^{\otimes G}
  = \und{T}(G/e)^{\otimes G}  \ra \und{T}(G/e). \]
We may consider the diagonal action of $G$ on $N_e^Gi_e^G\und{T}$, defined such that for all $\gamma \in G$ and $t_g \in \und{T}(G/e)$: 
\begin{equation} \label{eq:diagonal} \gamma(\bigotimes_{g \in G} t_g) = \bigotimes_{\gamma g \in G}
  \gamma(t_g). \end{equation}
In contrast, the flip action is defined by
\begin{equation} \label{eq:flip} \gamma(\bigotimes_{g \in G} t_g) = \bigotimes_{\gamma g \in G}
  t_g. \end{equation}

\begin{prop} \label{prop:counit}
For $N_e^Gi_e^G\und{T}$ with the diagonal $G$-action, the counit map  
 at level $G/e$ is the multiplication $\mu$ 
 of the commutative ring $\und{T}(G/e)$.

 If we use the flip action on $N_e^Gi_e^G\und{T}$, then at level $G/e$  the
 counit map sends $\bigotimes_{g \in G} t_g$ with $t_g \in \und{T}(G/e)$ to
 $\prod_{g \in G} gt_g$. 
  
\end{prop}
We denote the counit map for the diagonal action on
$N_e^Gi_e^G\und{T}$ by $\varepsilon^d$ and the one for the flip action by 
$\varepsilon^f$. 

\begin{proof}
Note that the norm-restriction adjunction ensures that maps of
$G$-Tambara functors from $N_e^Gi_e^G\und{T}$ to $\und{T}$ are in
natural bijection with maps of $e$-Tambara functors from $i_e^G\und{T}
= \und{T}(G/e)$ to 
itself. 

First we note that $\varepsilon^f$ and $\varepsilon^d$ are adjoint to the
identity map on $\und{T}(G/e)$: The identity map determines the value
on the $g=e$-tensor factor in $\bigotimes_{g \in G} \und{T}(G/e)$. If
we consider the 
norm-restriction with the flip action,  then we can write an
arbitrary element $\bigotimes_{g \in G} t_g$ as
\[ \bigotimes_{g \in G} t_g = \prod_{g \in G}g(t_g \otimes 1 \otimes
  \ldots \otimes 1)\] and hence obtain 
\[ \varepsilon^f_e(\bigotimes_{g \in G} t_g) = \prod_{g \in
    G}g\varepsilon^f(t_g \otimes 1 \otimes \ldots \otimes 1) =
  \prod_{g \in G} gt_g. \] 
In constrast, for the diagonal action we have to write $\bigotimes_{g
  \in G} t_g$ as 
\[ \bigotimes_{g \in G} t_g = \prod_{g \in G}g(g^{-1}t_g \otimes 1
  \otimes \ldots \otimes 1)\] and therefore $\varepsilon^d$ as the
adjoint to the identity map on $\und{T}(G/e)$ is 
\begin{align*}
  \varepsilon^d_e(\bigotimes_{g \in G} t_g) & =
\varepsilon^d_e(\prod_{g \in G}g(g^{-1}t_g \otimes 1 \otimes \ldots
                                              \otimes 1)) = \prod_{g
                                              \in G}g
                                              \varepsilon^d(g^{-1}t_g
                                              \otimes 1 \otimes \ldots
                                              \otimes 1) \\ 
  & = \prod_{g \in G} g(g^{-1}t_g) = \prod_{g\in G} t_g. \end{align*}

\end{proof}


\begin{rem}
The two different choices of $G$-actions on $N_e^Gi_e^G\und{T}$
correspond to two different choices of $G$-actions on $N_e^Gi_e^GA$
where $A$ is any 
commutative genuine $G$-ring spectrum, for instance $H\und{T}$.  One action just permutes the
smash factors in $\bigwedge_{G} A$ and the other one combines this
permutation action with the $G$-action on $A$. 
\end{rem}

\begin{prop} \label{prop:psi}
There is an isomorphism $\Psi$ of $G$-Tambara functors between
$N_e^Gi^G_e\und{T}$ with the flip $G$-action and $N_e^Gi^G_e\und{T}$ with
the diagonal $G$-action, such that the diagram
\[\xymatrix{
N_e^Gi^G_e\und{T} \ar[rd]_{\varepsilon^f}\ar[rr]^\Psi& & N_e^Gi^G_e\und{T} \ar[dl]^{\varepsilon^d} \\
& \und{T} &   }\]
commutes.  
\end{prop}
\begin{proof}
By adjunction,  maps of $G$-Tambara functors from
$N_e^Gi^G_e\und{T}$ to itself are in natural bijection with morphisms
of $e$-Tambara functors from $i_e^G\und{T} = \und{T}(G/e)$ to
$i_e^GN_e^Gi_e^G\und{T} = \bigotimes_{g \in G}\und{T}(G/e)$.

We define the adjoint of $\Psi$,
\[ \mathrm{ad}(\Psi) \colon \und{T}(G/e) \ra \bigotimes_{g \in
    G}\und{T}(G/e)\] as
\[ \mathrm{ad}(\Psi)(t) := \bigotimes_{g \in G} t_g \text{ with } t_g
  = \begin{cases} t, & \text{ if } g=e, \\ 1, &
                                    \text{otherwise}. \end{cases} \]   
On the level of adjoints we have to check that the composite
$i_e^G\varepsilon^d \circ \mathrm{ad}(\Psi)$ 
is the identity map which is the adjoint of $\varepsilon^f$, and indeed
\[ i_e^G\varepsilon^d \circ \mathrm{ad}(\Psi)(t) =
  i_e^G\varepsilon^d(t \otimes 1 \otimes \ldots \otimes 1) = t. \] 
The map $\Psi$ itself is then given at $G/e$  by 
  \begin{align*}
    \Psi(\bigotimes_{g \in G} t_g) & = \Psi(\prod_{g \in G} g(t_g \otimes 1 \otimes \ldots \otimes 1)) \\
                                  & = \prod_{g \in G} g \mathrm{ad}(\Psi)(t_g) \\
                                  & = \prod_{g \in G} g(t_g \otimes 1
                                    \otimes \ldots \otimes 1) \\ 
    & = \bigotimes_{g \in G} gt_g. 
\end{align*}
This is visibly an isomorphism.
\end{proof}

As $N_e^Gi_e^G\und{T}$ with the flip action is isomorphic to
$N_e^Gi_e^G\und{T}$ with the diagonal action and as this isomorphism
is compatible with the counit maps we can deduce the following
result. Here, $S^\sigma$ is a simplicial model of the $1$-point
compactification of the real sign representation with $(S^\sigma)_0 =
C_2/C_2 \sqcup C_2/C_2$ and $(S^\sigma)_1 =   C_2/C_2 \sqcup C_2/e
\sqcup C_2/C_2$:
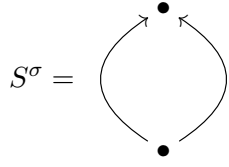
\[\xymatrix@R=0.5cm{& \bullet \\ S^\sigma = &  \\ & \bullet \ar@/^5ex/[uu] \ar@/_5ex/[uu]}\]
\begin{prop}
  Let $\und{R}$ be a $C_2$-Tambara functor, then there is an isomorphism of
  simplicial $C_2$-Tambara functors 
  \[ \cL_{S^\sigma}^{C_2,f}(\und{R}) \cong \cL_{S^\sigma}^{C_2}(\und{R})\]
  where on the left hand side we consider the flip action on
  norm-restriction terms and on the right hand side we consider the
  diagonal action as in \cite{lrz-gloday}. 
\end{prop}
\begin{proof}
  In \cite[(7.4)]{lrz-gloday} we identified $\cL_{S^\sigma}^{C_2}(\und{R})$ with
  the two-sided bar construction $B(\und{R}, N_e^{C_2}i_e^{C_2}\und{R}, \und{R})$
  where the norm-restriction term carries the diagonal
  action. However, the isomorphism $\Psi$ from Proposition
  \ref{prop:psi} induces an isomorphism of simplicial $C_2$-Tambara functors
  \[ B(\id, \Psi, \id) \colon  B(\und{R}, N_e^{C_2}i_e^{C_2}\und{R},
    \und{R}) \ra B(\und{R}, N_e^{C_2}i_e^{C_2}\und{R}_d, \und{R})\]
  when in the left bar construction we use the flip action on
  $N_e^{C_2}i_e^{C_2}\und{R}$ and in the right bar construction we use
  the diagonal action 
  because $\Psi$ is compatible with the augmentation maps, the multiplication on $N_e^{C_2}i_e^{C_2}\und{R}$ and the unit maps. 
\end{proof}

The above result ensures that in the particular case of $S^\sigma$ the
choice of $C_2$-action on the norm-restriction terms does not
matter. However, in other examples, for instance an $S^1$ with a 
$C_n$-rotation action, the choice \emph{does} make a difference, see
for instance Example \ref{ex:twistedcyclicnerve}.

Let $\und{R}$ be an $H$-Tambara functor. On $N_H^G\und{R} =
\und{\pi}_0^G(\bigwedge_{G/H} H\und{R})$ one considers the
tensor-induction action. This is explicitly described for instance in
\cite[pp.~40,41]{cary} . 
For any choice of transversal of $H$ in $G$, $g_1,
\ldots, g_n$ and for every $g$ and all $1 \leq i \leq n$ in $G$
there is a unique $1 \leq j_i \leq n$ with  $g_{j_i}^{-1}gg_i \in H$
and the action is  
given by 
\begin{equation}\label{eqn:HHR}
 g \big( \bigotimes_{g_iH} r_i \big)
  = \bigotimes_{g_{j_i} H} (g_{j_i}^{-1}gg_i) \cdot r_i. 
  \end{equation}
  This comes from the decomposition of $G$ as
  $G=\bigsqcup _{g_iH\in G/H} g_iH$ and
    the observation that if $g_{j_i}^{-1}gg_i \in H$, $g\cdot g_i h =
    g_{j_i} (g_{j_i}^{-1}gg_i h) \in g_{j_i}H$.  Note that if we pick a different
    transversal, there is an explicit isomorphism between
    the norm constructed using the first transversal and the norm
    constructed using the second transversal.   

For any subgroup $e < H < G$ and with the tensor induction action on
$N_e^Gi_e^H\und{R}$,
the flip action on $N_e^Hi_e^H\und{R}$ and the tensor induction action on
$N_H^G$ there is an isomorphism of $G$-Tambara functors 
\begin{equation} \label{eq:xi}
  \xi \colon N_e^Gi_e^H\und{T} \cong N_H^GN_e^Hi_e^H\und{T}.
\end{equation}
On the level of spectra this isomorphism is induced by reordering smash factors
and evaluated on the free orbit it sorts the $G$-fold tensor product
of $\und{T}(G/e)$ into $\bigotimes_{G/H}\bigotimes_H \und{T}(G/e)$. As
we don't consider the diagonal action, the map $\xi$ is equivariant.

\begin{ex}
  Let us make explicit why we cannot take the diagonal action on
  $N_e^Hi_e^H\und{R}$: 
  
Let $G$ be $\Sigma_3$ and let $H$ be the subgroup generated by the 
transposition $\tau=(1,2)$. We choose as coset representatives for $H$ in
$\Sigma_3$ the elements  $g_1 = \id, g_2 = (1,3), g_3 = (2,3)$. 

At the free level, $\xi_e$ is a map
\[ \xi_e \colon (N_e^{\Sigma_3}i_e^H\und{R})(\Sigma_3/e) =
  \bigotimes_{\Sigma_3} \und{R}(H/e) 
  \ra \bigotimes_{\id, (1,3), (2,3)} \ \ \bigotimes_{\id, (1,2)} \und{R}(H/e) = (N_H^{\Sigma_3}N_e^Hi_e^H\und{R})(\Sigma_3/e) \]  
and it sends a generator $r_{id} \otimes r_{(1,2)} \otimes r_{(1,3)}
\otimes r_{(2,3)} \otimes r_{(1,2,3)} \otimes r_{(1,3,2)}$ to
\[ \xi_e(r_{id} \otimes r_{(1,2)} \otimes r_{(1,3)}
\otimes r_{(2,3)} \otimes r_{(1,2,3)} \otimes r_{(1,3,2)}) = (r_{id} \otimes r_{(1,2)}) \otimes (r_{(1,3)}
\otimes r_{(1,2,3)}) \otimes (r_{(2,3)} \otimes r_{(1,3,2)}). \]
With the diagonal action on $N_e^Hi_e^H\und{R}(H/e)$ inside the term  $N^{\Sigma_3}_HN_e^Hi_e^H\und{R}(H/e)$ for instance
$(2,3) \circ \xi_e$ would yield
\begin{equation}
\label{eq:1}
(r_{(2,3)} \otimes r_{(1,3,2)}) \otimes (\overline{r}_{(1,2,3)} \otimes \overline{r}_{(1,3)}) \otimes (r_{id} \otimes r_{(1,2)})
\end{equation}
where $\overline{a} = \tau(a)$.

On the other hand, first applying $(2,3)$ and then $\xi_e$ would give
\[ \xi_e \circ (2,3)(r_{id} \otimes r_{(1,2)} \otimes r_{(1,3)}
\otimes r_{(2,3)} \otimes r_{(1,2,3)} \otimes r_{(1,3,2)}) = 
(r_{(2,3)} \otimes r_{(1,3,2)}) \otimes (r_{(1,2,3)} \otimes r_{(1,3)}) \otimes (r_{id} \otimes r_{(1,2)})\]
because $i^H_e\und{R}$ does not carry any $H$-action.

Hence if we use the diagonal action, then the map $\xi_e$ is  not
equivariant and hence $\xi$ is not an isomorphism of $G$-Tambara
functors.

\medskip 
Even if we tried
to remember the $H$-action on $i_e^H\und{R}$, by sending $r_\sigma$ in
component $\sigma$ to $r_{(2,3)\sigma}$ in component $\sigma$ and then
to itself, if $(2,3)\sigma = g_i$ and to
$\overline{r}_{(2,3)\sigma}$ if $(2,3)\sigma = g_i (1,2)$ for some
representative $g_i$, then we would get 
\[ (2,3)r_{id} \otimes r_{(1,2)} \otimes r_{(1,3)}
\otimes r_{(2,3)} \otimes r_{(1,2,3)} \otimes r_{(1,3,2)} = r_{(2,3)}
\otimes \overline{r}_{(1,3,2)} \otimes \overline{r}_{(1,2,3)} 
\otimes r_{\id} \otimes r_{(1,3)} \otimes \overline{r}_{(1,2)} \]
and applying $\xi_e$ then yields
\[(r_{(2,3)} \otimes \overline{r}_{(1,3,2)}) \otimes (\overline{r}_{(1,2,3)} \otimes r_{(1,3)}) \otimes
  (r_{id} \otimes \overline{r}_{(1,2)}),\]
which is different from \eqref{eq:1}.
\end{ex}

For our Loday construction for an $H$-Tambara functor later, we want to 
assign $N_H^G\und{R}$ to the orbit $G/H$. Therefore we have to show
that equivariant self-maps of $G/H$ induce well-defined self-maps on
$N_H^G\und{R}$.  

\begin{lem}\label{lem:Weyl}
  Let $H<G$ be a subgroup and let $\und{R} $ be an $H$-Tambara functor.  Then the $G$-Tambara functor $N_H^G\und{R}$ carries a $W_G(H)$-action. This action is
  natural in the $H$-Tambara functor $\und{R}$. 
\end{lem}  
\begin{proof}
Let
$\gamma \in N_G(H)$ and let $[\gamma] \in W_G(H)$ be its Weyl class.   
We choose a transversal $(g_1,\ldots, g_n)$ for $H$ in $G$. Note that
$(g_1\gamma^{-1}, \ldots, g_n\gamma^{-1})$ is then also a transversal
for $H$ in $G$. 

The Weyl group $W_G(H)$ acts on $G/H$ by sending a coset $gH$ to
\[ [\gamma](gH) = g\gamma^{-1}H. \]
This induces an action on 
\[ N_H^G\und{R}(G/e) = \bigotimes_{g_iH} \und{R}(H/e)\]
by sending an element $\bigotimes_{g_iH} r_i \in \bigotimes_{g_iH}
\und{R}(H/e)$ to
\[ [\gamma](\bigotimes_{g_iH} r_i) = \bigotimes_{g_i\gamma^{-1}H}
  r_i. \]
We show that this action is $G$-equivariant with respect to the tensor
induction action from \eqref{eqn:HHR}, \ie, that the following
diagram commutes for all $g \in G$ and $[\gamma] \in W_G(H)$:
\[ \xymatrix{
\bigotimes_{g_iH} \und{R}(H/e) \ar[rr]^{[\gamma]} \ar[d]_g & &
\bigotimes_{g_i\gamma^{-1}H} \und{R}(H/e) \ar[d]^g\\
\bigotimes_{g_iH} \und{R}(H/e) \ar[rr]^{[\gamma]} && \bigotimes_{g_i\gamma^{-1}H}
\und{R}(H/e)   } \]
By definition
\[ [\gamma]\Big(g\Big(\bigotimes_{g_iH} r_i\Big)\Big) =
  [\gamma]\Big(\bigotimes_{g_{j_i}H} (g_{j_i}^{-1}gg_i)r_i\Big) = 
\bigotimes_{g_{j_i}\gamma^{-1}H} (g_{j_i}^{-1}gg_i)r_i. \]
Here, $j_i$ is the unique index such that $gg_iH = g_{j_i}H$. On the
other hand
\[ g\Big([\gamma]\Big(\bigotimes_{g_iH} r_i\Big)\Big) =
  g\Big(\bigotimes_{g_i\gamma^{-1}H} r_i\Big) =
  \bigotimes_{g_{k_i}\gamma^{-1}H} (g_{k_i}^{-1}gg_i)r_i, \] 
where $k_i$ is the unique index concerning the transversal
$(g_1\gamma^{-1}, \ldots, g_n\gamma^{-1})$ such that $gg_i\gamma^{-1}H =
g_{k_i}\gamma^{-1}H$. But then $j_i = k_i$.  

As the Weyl-group action is $G$-equivariant on $N_H^G(\und{R})(G/e)$ we
obtain that on spectrum level this corresponds to  a $G$-equivariant Weyl-group
action on $N_H^GH\und{R}$ and this in turn induces a well-defined
Weyl-group action on the $G$-Tambara functor
$\und{\pi}_0^GN_H^G(H\und{R}) = N_H^G(\und{R})$. 

For the claim about naturality assume that $f \colon \und{S} \ra \und{R}$ is a
morphism of $H$-Tambara functors. We have to show that for all $[\gamma] \in
W_G(H)$ the diagram
\[\xymatrix{N_H^G \und{S} \ar[rr]^{[\gamma]} \ar[d]_{N_H^G(f)} & &
    N_H^G \und{S} \ar[d]^{N_H^G(f)}\\
  N_H^G \und{R} \ar[rr]^{[\gamma]}  & &
    N_H^G \und{R} }   \]
commutes. Again, we use the norm-restriction adjunction, so the above
requirement translates to the commutativity of the diagram
\[ \xymatrix{ \und{S} \ar[rr]^{ad[\gamma]} \ar[d]_f& & i_H^GN_H^G\und{S}
\ar[d]^{i_H^GN_H^G f} \\    
\und{R} \ar[rr]^{ad[\gamma]} & & i_H^GN_H^G\und{R}. } \]
John Ullman shows \cite[Theorem 5.2]{ullman} that maps between $H$-Tambara
functors correspond bijectively to the maps between the corresponding
Eilenberg-MacLane spectra. Therefore, we can check the above commutativity on
spectrum level and for this it is enough to check the claim for Tambara
functors at the free level. 

Evaluated at $H/e$ the above diagram is
\[ \xymatrix{ \und{S}(H/e) \ar[rr]^{ad[\gamma]} \ar[d]_{f(H/e)} & &
    \bigotimes_{g_iH} \und{S}(H/e)
\ar[d]^{(i_H^GN_H^G f)(H/e)} \\    
\und{R}(H/e) \ar[rr]^{ad[\gamma]} & & \bigotimes_{g_iH}\und{R}(H/e) } \]
for a choice of a transversal $(g_1,\ldots, g_n)$ of $H$ in $G$ with $g_1=e$
and  where $ad[\gamma]$ maps an $s \in \und{S}(H/e)$ to the tensor product
$\bigotimes_{g_i\gamma^{-1}H} s_{g_i\gamma^{-1}}$ with
\[ s_{g_i\gamma^{-1}} = \begin{cases} s, & g_i = g_1 \\ 1, & g_i \neq g_1.
  \end{cases} \] 
Then the commutativity of the diagram follows by direct inspection.

    \end{proof}

\bigskip 

It turns out that the Weyl group action in Lemma \ref{lem:Weyl} comes from the restriction
to $N_G(H)$ of a $G$-action because $G$ acts by conjugation on the collection of conjugates of a subgroup $H$.  The normalizer
subgroup $N_G(H)$ fixes $H$ in this action, and then acts on $G/H$; in this $N_G(H)$-action, $H$ acts trivially on $G/H$.
Similarly, $\gamma N_G(H) \gamma^{-1}$ fixes $\gamma H\gamma^{-1}$ and then acts on $G/\gamma H\gamma^{-1}$; in this $\gamma N_G(H) \gamma^{-1}$-action, $\gamma H\gamma^{-1}$ acts trivially on $G/\gamma H\gamma^{-1}$.  This can be upgraded to our Tambara construction:

\begin{lem}\label{lemma:switchingtoconjugates}
   Let $H<G$ be a subgroup and let $\und{R} $ be an $H$-Tambara functor.  Then any $\gamma\in G$ defines an isomorphism of $G$-Tambara functors 
   \[N_H^G\und{R} \to N_{\gamma H\gamma^{-1}}^G( c_{\gamma^{-1}}^*(\und{R})).\]
     Here, $c_{\gamma^{-1}}$ is the conjugation automorphism of $G$, $c_{\gamma^{-1}}(g)= \gamma^{-1} g \gamma$ for all $g\in G$, and so $H= c_{\gamma^{-1}}(\gamma H\gamma^{-1})$, and the pullback of Tambara functor is as described in Lemma \ref{lem:varphir}.
\end{lem}  
\begin{proof}
As before, we choose a transversal $(g_1,\ldots, g_n)$ for $H$ in $G$ and note that
$(g_1\gamma^{-1}, \ldots, g_n\gamma^{-1})$ is a transversal
for $\gamma H\gamma^{-1}$ in $G$: if $G=\bigsqcup_{i=1}^n g_i H$, then                        
\[G=G\cdot \gamma^{-1} =   \bigsqcup_{i=1}^n g_i H  \gamma^{-1} =  \bigsqcup_{i=1}^n (g_i  \gamma^{-1}) (\gamma H  \gamma^{-1} ).\]
So now we map 
\[ N_H^G\und{R}(G/e) = \bigotimes_{g_iH} \und{R}(H/e)\]
to \[ N_{\gamma H\gamma^{-1}}^G( c_{\gamma^{-1}}^*(\und{R}))
(G/e) = \bigotimes_{(g_i  \gamma^{-1}) (\gamma H  \gamma^{-1} )} ( c_{\gamma^{-1}}^*(\und{R}))  (\gamma H  \gamma^{-1} /e)
= \bigotimes_{(g_i  \gamma^{-1}) (\gamma H  \gamma^{-1} )}  \und{R}(H/e)
\]
by sending an element $\bigotimes_{g_iH} r_i \in \bigotimes_{g_iH}
\und{R}(H/e)$ to
\[ \gamma (\bigotimes_{g_iH} r_i) = \bigotimes_{(g_i  \gamma^{-1}) (\gamma H  \gamma^{-1} )}
  r_i. \] 
As above, to show that this action is $G$-equivariant with respect to the tensor
induction action from \eqref{eqn:HHR}, we need to show that the following
diagram commutes for all $g, \gamma \in G$:
\[ \xymatrix{
\bigotimes_{g_iH} \und{R}(H/e) \ar[rr]^{\gamma\quad \quad\quad} \ar[d]_g & &
\bigotimes_{(g_i  \gamma^{-1}) (\gamma H  \gamma^{-1} )} ( c_{\gamma^{-1}}^*(\und{R}))  (\gamma H  \gamma^{-1} /e) \ar[d]^g\\
\bigotimes_{g_iH} \und{R}(H/e) \ar[rr]^{\gamma\quad \quad\quad} &&
 \bigotimes_{(g_i  \gamma^{-1}) (\gamma H  \gamma^{-1} )} ( c_{\gamma^{-1}}^*(\und{R}))  (\gamma H  \gamma^{-1} /e). }\]
By definition
\begin{equation}\label{first}
 \gamma \Big(g\Big(\bigotimes_{g_iH} r_i\Big)\Big) =
  \gamma\Big(\bigotimes_{g_{j_i}H} (g_{j_i}^{-1}gg_i)r_i\Big) = 
\bigotimes_{(g_{j_i}  \gamma^{-1}) (\gamma H  \gamma^{-1} )}  (g_{j_i}^{-1}gg_i)r_i 
\end{equation}
where  $j_i$ is the unique index such that $gg_iH = g_{j_i}H$. On the
other hand
\begin{equation}\label{second}
 g\Big(\gamma\Big(\bigotimes_{g_iH} r_i\Big)\Big) =
  g\Big( \bigotimes_{(g_i  \gamma^{-1}) (\gamma H  \gamma^{-1} )}
  r_i \Big) =
  \bigotimes_{(g_{k_i} \gamma^{-1}) (\gamma H  \gamma^{-1} )}
((g_{k_i}\gamma^{-1})^{-1}gg_i\gamma^{-1})r_i, 
\end{equation}
where $k_i$ is the unique index in the transversal
$(g_1\gamma^{-1}, \ldots, g_n\gamma^{-1})$ for $\gamma H  \gamma^{-1} $ such that 
\[  gg_i H \gamma^{-1} = gg_i\gamma^{-1}(\gamma H  \gamma^{-1} ) =
g_{k_i}\gamma^{-1}(\gamma H  \gamma^{-1} )= g_{k_i} H \gamma^{-1}
,\] 
so again $j_i = k_i$.  Note that 
in Equation \eqref{first}, we regard $r_i$ as an element in $\und{R}(H/e)$ and act on it by $g_{j_i}^{-1}gg_i$ whereas in Equation \eqref{second}, we regard $r_i$ as an element in
$c_{\gamma^{-1}}^*(\und{R}) (\gamma H  \gamma^{-1}/e) $ and act on it by 
$\gamma g_{k_i}^{-1}gg_i\gamma^{-1}= \gamma g_{j_i}^{-1}gg_i\gamma^{-1}$. This
agrees with acting on $r_i$ as an element in $\und{R}(H/e)$ via $c_{\gamma^{-1}}(\gamma g_{j_i}^{-1}gg_i\gamma^{-1}) = g_{j_i}^{-1}gg_i$.

On the spectrum level this corresponds to  a $G$-equivariant 
map  $N_H^GH\und{R}\to N_{\gamma H\gamma^{-1}}^GH( c_{\gamma^{-1}}^*(\und{R}))$ and this in turn induces the desired map
on $\und{\pi}_0$ of these as asserted in the statement of the lemma.

The map is multiplicative in $\gamma$, so its inverse is the map 
 $N_{\gamma H\gamma^{-1}}^G( c_{\gamma^{-1}}^*(\und{R}))
 \to  N_H^G\und{R} $
 induced by $\gamma^{-1}$.
\end{proof}

\bigskip 
If we start with a $\varphi(H)$-Tambara functor $\und{R}$ as in Lemma
\ref{lem:varphir}, then we will need morphisms of $G$-Tambara functors
$N_e^Gi_e^{\varphi(H)}\und{R} \ra
N_{\varphi(H)}^G\und{R}$ and $N_e^Gi_e^{\varphi(H)}\und{R} \ra
N_H^G(\varphi^*\und{R})$. The map $N_e^Gi_e^{\varphi(H)}\und{R} \ra
N_{\varphi(H)}^G\und{R}$ is the composite  
\begin{equation} \label{eq:counitgh'}
 \xymatrix@1{ N_e^Gi_e^{\varphi(H)}\und{R} \ar[r]^(0.4)\xi &
   N_{\varphi(H)}^GN_e^{\varphi(H)} i_e^{\varphi(H)}\und{R} \ar[rr]^(0.6){N_{\varphi(H)}^G(\varepsilon^f)}
   & &  N_{\varphi(H)}^G\und{R}}, \end{equation}
whereas the map
$N_e^Gi_e^{\varphi(H)}(\und{R}) \ra N_H^G(\varphi^*\und{R})$ 
is the composite
\begin{equation}\label{eq:counitgh}
 \xymatrix@1{ N_e^Gi_e^{\varphi(H)}\und{R} \ar[r]^(0.45)\xi &
   N_{H}^GN_e^{H} i_e^{\varphi(H)} \und{R}
   \ar[rr]^(0.55){N_{H}^G(\varepsilon^f_\varphi)} 
   & &  N_{H}^G\varphi^*\und{R}}.  
\end{equation}
Here $\varepsilon^f_\varphi\colon N_e^Hi_e^{\varphi(H)}(\und{R}) \ra
\varphi^*\und{R}$ is determined by its adjoint
$i_e^{\varphi(H)}(\und{R}) \ra i_e^H\varphi^*\und{R}$
\[ \mathrm{ad}(\varepsilon^f_\varphi) \colon i_e^{\varphi(H)}\und{R} =
  \und{R}(\varphi(H)/e) \ra i_e^H\varphi^*\und{R} =
  (\varphi^*\und{R})(H/e) = \und{R}(\varphi(H)/e), \]
the identity map on $\und{R}(\varphi(H)/e)$. 


\section{Loday constructions for Tambara functors with respect to  isotropy subgroups} \label{sec:isotropy}
Our goal is to reduce the requirements for a Tambara functor, so that
the equivariant Loday construction for a finite $G$-simplicial set is
still defined.

\subsection{Trivial isotropy} 
We start with an extreme case:
\begin{prop} \label{prop:free}
Assume that $X$ is a finite simplicial
free $G$-set, \ie, that every $X_n$ is of the form $\bigsqcup_{i=1}^{m_n} G/e$ for
some $1 \leq m_n < \infty$. Then $\cL_X^{G;e}(R)$ can be defined for every
commutative $G$-ring $R$ and the $G$-action on $(N_e^GR)(G/e)$ can be chosen to be
the diagonal or the flip-action.  
\end{prop}
\begin{proof}
As $X$ is free with $X_n$ as above, then we can set 
\[ \cL_X^{G;e}(R)_n := \bigbox_{i=1}^{m_n} N_{e}^G R.\]
Here, $N_e^GR$ is a $G$-Tambara functor and in this case we have the
choice whether we just use the flip-action of $G$ on $N_e^GR$ or the
diagonal $G$-action on 
$N_e^GR$ where $G$ permutes the box-product factors in the norm and acts
coordinatewise on the tensor factors of $R$.

For every $f \colon [m] \ra [n] \in \Delta$ the only possible induced
maps $X(f) \colon X_n \ra X_m$ are built out of inclusions
$\varnothing \ra G/e$, fold maps, permutations of summands and
equivariant self-maps of $G/e$.

Inclusions $\varnothing \hookrightarrow G/e$ insert the unit $\und{A}
\ra N_e^GR$ of the $G$-Tambara functor $N_e^GR$, fold maps induce
multiplication on $N_e^GR$, permutations give rise to permutations of
box-product factors and equivariant self-maps of $G/e$ induce a
$G$-action on $N_e^GR$. No matter whether we choose the diagonal or the flip action on $N_e^GR$, all the remaining structure is compatible with either choice. 
Thus we only need the commutative $G$-ring $R$ in order to define the
equivariant Loday construction $\cL_X^{G;e}(R)$. 
\end{proof}

\begin{rem}
Note that  we allow the trivial $G$-action on
$R$, so any commutative ring will do.  

In the above result we assign $N_e^GR$ to a free orbit $G/e$. In the ordinary
equivariant Loday construction we would denote this by $G/e \otimes \und{T} = N_e^Gi_e^G\und{T}$ 
if $\und{T}$ is a $G$-Tambara functor. But here, the input is an $e$-Tambara
functor (which is nothing but a commutative ring) or a commutative ring with
$G$-action. In this setting a notation like
  \[ G\otimes_e R := N_e^GR\]
  is more appropriate in order to make clear what the structure of the
  input is. 
\end{rem}

In the free case we are free to make either choice, the flip or the
diagonal action. The resulting Loday constructions, however, might differ:

\begin{ex} \label{ex:twistedcyclicnerve}
Consider the cyclic group of order $n$, $C_n = \langle \gamma \mid \gamma^n=e\rangle$, and a simplicial model
of $S^1$ with the rotation action, $S^1_{rot}$:
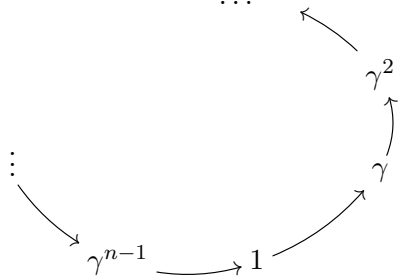
\[ \xymatrix@C=0.7cm@R=0.5cm{
 &  &\ldots \phantom{bla} & \\
& && \gamma^2 \ar@/_1ex/[ul]\\
\, \, \vdots \ar@/_1ex/[dr]& && \gamma \ar@/_1ex/[u]\\
&\gamma^{n-1} \ar@/_1ex/[r] & 1 \ar@/_1ex/[ur]&
  }\]

The corresponding simplicial set $S^1_{rot}$ is free in every degree
and we identified its structure in \cite[\S 7.1]{lrz-gloday}. Let $\und{R}$
be a $C_n$-Tambara functor. If we
take the diagonal action on $N_e^{C_n}i_e^{C_n}\und{R}$, then
$\cL^{C_n}_{S^1_{rot}}(\und{R})$ is isomorphic to the twisted cyclic nerve
of $N_e^{C_n}i_e^{C_n}(\und{R})$ in the sense of \cite{bghl}.

On the other hand, just 
acting by the flip action gives $\cL^{C_n,f}_{S^1_{rot}}(\und{R})(C_n/e) \cong
\mathrm{sd}_n\cL_{S^1} (\und{R}(C_n/e))$, where the latter is the $n$-fold subdivision of the cyclic bar construction of the commutative ring $\und{R}(C_n/e)$.
In particular, in this simplicial object the last face map is \emph{not} twisted and its homotopy groups are the ordinary Hochschild homology groups of
$\und{R}(C_n/e)$. We discussed this case in \cite[Theorem 7.4]{lrz-gloday}. 

Due to Proposition \ref{prop:free} we can define $\cL^{C_n;e}_{S^1_{rot}}(R)$
for every $C_n$-commutative ring $R$. If the $C_n$-action on $R$ is non-trivial, then the Loday construction with the diagonal action on $N_e^{C_n}(R)$ differs
from the one with the flip-action. 

\end{ex}

The above example is a special case of a Cayley graph.  Let $G$ be a
finite and finitely presented group of the form $G = \langle s_1, 
\ldots, s_n \mid r_1, \ldots, r_m\rangle$ with $s_i \neq e$ for all
$i$ and by the usual abuse of notation we denote by $\cC_G$ the 
Cayley graph of $G$ with respect to this presentation. Here, we take
the full Cayley graph of $G$ in the sense that we do not simplify
loops of length two corresponding to involutions to an unoriented
edge, but we keep the loop. With this convention we can consider a
simplicial model of the Cayley graph and by another abuse of notation we
denote this finite simplicial $G$-set also by $\cC_G$. Then the action
of $G$ on $\cC_G$ is free: The vertices of the Cayley graph correspond
to elements $g$ in the group $G$ and the group acts from the left on
the vertices. Thus $(\cC_G)_0 = G/e$. An edge
  $\xymatrix@1{g \ar[r]^{\cdot s_j} & gs_j}$ in the Cayley graph
  corresponds to a $1$-simplex and the corresponding orbit is free as
  well. All higher dimensional simplices in $\cC_G$ are degenerate. 

An immediate consequence of Proposition \ref{prop:free} is the
following result.
\begin{cor}
  Let $\cC_G$ be the simplicial model of a Cayley graph of a finite
  group with a chosen finite presentation as above. Then the $G$-Loday
  construction $\cL_{\cC_G}^{G;e}(R)$ can be defined for all commutative
  $G$-rings $R$. 
\end{cor}

\subsection{One non-trivial isotropy group (and its conjugates)}\label{sec:oneisotropy}
We now consider the case where a finite simplicial $G$-set $X$ has
non-trivial isotropy. 

If $e \neq H < G$ is an isotropy subgroup  of a
point $x\in X$, then $gHg^{-1}$ is the isotropy group of $g\cdot x$. 
So if we have an orbit $G\cdot y$ isomorphic to $G/gHg^{-1}$,
we could pick a different representative and consider it the orbit of $x=g^{-1}\cdot y$
and think of that orbit as isomorphic to $G/H$.  
This is not surprising since $G/H\cong G/gHg^{-1} $ as $G$-sets.  So on the
one hand, once we get $H$ as an isotropy group, we get all of its conjugates
as isotropy groups as well, but  on the other hand, dealing with the orbit
$G/H$ covers all of these isotropy cases.

Moreover, if we have a map between two $G$-sets $G/H\to G/gHg^{-1} $
and we identify $G/gHg^{-1} \cong G/H$, we can just think of our map as a $G$-map $G/H\to G/H$; these are given by
the action of the Weyl group.  

\begin{thm} \label{thm:conjugateisotropy}
Let $H\neq e$ be a subgroup of $G$, and assume that a finite
simplicial $G$-set $X$ has only   $e$, $H$ and its conjugates $
gHg^{-1}$, $g\in G$ as
isotropy subgroups of its elements.  Then by the above comments, every $X_n$ can be written as
\[ X_n = \bigsqcup_{i \in E_e} G/e \sqcup \bigsqcup_{x \in E_H} G/H
\]
for some finite sets $E_e$ and $E_H$. 
Let $\und{R}$ be an $H$-Tambara functor.  Then
\[ \cL_X^{G;H}(\und{R})_n := \Big(\bigbox_{i \in E_e}
  N_e^Gi_e^H\und{R}\Big) \Box  \Big(\bigbox_{x \in E_H}
  N_H^G\und{R}\Big) \]
gives rise to a well-defined simplicial $G$-Tambara functor 
$\cL_X^{G;H}(\und{R})$ if one uses the flip action on the norm-restriction terms 
$N_e^{H}i_e^{H}\und{R}$.  
\end{thm}

\begin{proof}
  Again, we have to describe the effect of the simplicial structure maps and show that the simplicial identities hold.  

For every $f \colon [m] \ra [n] \in \Delta$ the possible induced
maps $X(f) \colon X_n \ra X_m$ are built out of inclusions
$\varnothing \ra G/e$ and $\varnothing \ra G/H$, fold maps $G/e \sqcup
G/e \ra G/e$ and $G/H \sqcup G/H \ra G/H$, permutations of summands,  
equivariant self-maps of $G/e$, $G/H$ and equivariant
surjections  $G/e \ra 
G/H$. 

As all
equivariant maps $G/e \ra G/H$ are a composite of a 
$G$-self-map on $G/e$ and the standard projection $\pi \colon G/e \ra
G/H$, we only cover the case of the standard projection. Note that degeneracies induce
 injections that are typically non-isomorphic and face maps induce 
 surjections that are typically non-isomorphic, thus those maps are composites
 of fold maps and projections $\pi$ combined with isomorphisms. 

The inclusions $j_e \colon \varnothing \ra G/e$ and $j_{H} \colon
\varnothing \ra G/H$ induce 
the unit maps of the commutative rings $\und{R}(G/e)$ and
$\und{R}(G/H)$ and fold maps induce multiplication maps in these
rings. Equivariant self maps induce the $G$-action on $\und{R}(G/e)$
and the $W_G(H)$-action  on $\und{R}(G/H)$  which was described in Lemma
\ref{lem:Weyl}.

The projection map $\pi
\colon G/e \ra G/H$ induces the map
\[ N_H^G(\varepsilon^f) \circ \xi \colon N_e^Gi_e^H\und{R} \cong
N_H^GN_e^Hi_e^H\und{R} \ra N_H^G\und{R}\]  and this is the crucial point
where we need to use the flip action on norms.

The naturality of the fold map $\nabla$ ensures that the diagram
\[ \xymatrix{
G/e \sqcup G/e \ar[rr]^{\pi \sqcup \pi} \ar[d]_{\nabla_e}&& G/H \sqcup
G/H \ar[d]^{\nabla_H} \\  
G/e \ar[rr]^\pi & & G/H
  }\]
commutes. Therefore, also higher iterations of fold maps and
projections commute such as 
\[ \xymatrix{G/H \sqcup G/e \sqcup G/e  \ar[rr]^{\id \sqcup \pi \sqcup
      \pi} \ar[d]_{\id \sqcup \nabla_e}& & G/H \sqcup G/H \sqcup G/H
    \ar[d]^{\id \sqcup \nabla_H} \\      
G/H \sqcup G/e  \ar[d]_{\id \sqcup \pi} & & G/H \sqcup G/H \ar[d]^{\nabla_H}\\
G/H \sqcup G/H \ar[rr]^{\nabla_H} & & G/H
}\]
and the corresponding diagram with three copies of $G/e$ to start
with. This and the associativity of fold maps ensures that the
composite of two face maps satisfies the simplicial identities. 

Composites of degeneracies correspond to iterated inclusions like
\[ \xymatrix{\varnothing \ar[r]^(0.3){j_k} & G/K = G/K \sqcup \varnothing
    \ar[rr]^{\id \sqcup j_U} & &  G/K \sqcup G/U} \]
with $K,U \in \{e,H\}$, so these satisfy the simplicial
identities.

The commutative diagrams $\xymatrix{\varnothing \ar[r]^{j_e} \ar[rd]_{j_H} & G/e \ar[d]^\pi \\
  & G/H}$ induces the diagram $\xymatrix{\und{A} \ar[rr]^{\eta_{N_e^Gi_e^H\und{R}}} \ar[drr]_{\eta_{N_H^G\und{R}}}& & N_e^Gi_e^H\und{R} \ar[d]^{N_H^G(\varepsilon^f) \circ \xi} \\
 & & N_H^G\und{R}. }$ 
This commutes because the Burnside Tambara functor $\und{A}$ is an initial object in the category of $G$-Tambara functors. 

Last but not least, the diagram  
\[  \xymatrix{G/K \sqcup G/K \ar[drrr]_{\nabla_K} & & \ar[ll]_{j_K \sqcup \id}
    \varnothing \sqcup G/K &  \ar@{=}[l]G/K \ar@{=}[d]\ar@{=}[r] & G/K \sqcup \varnothing  \ar[rr]^{\id \sqcup
    j_K} & & G/K \sqcup G/K  \ar[dlll]^{\nabla_K}\\
 & & & G/K &&&}   \]
for $K \in \{e,H\}$ induces a diagram of $G$-Tambara functors that expresses the
fact that the multiplications in the rings $N_e^Gi_e^H\und{R}$ and $N_H^G\und{R}$ have
a right- and left-sided unit. Hence the mixed simplicial identities are satisfied. 

\end{proof}

\begin{ex}
If a finite simplicial $G$-set $X$ has isotropy subgroups $e$ and $H$ and if $H$
is normal in $G$, then we do not need to worry about picking a good representative for each orbit. We can choose any representative and the only occurring
orbit types are $G/e$ and $G/H$. Examples of such $G$-simplicial sets are Cayley graphs for
$G/H$ viewed as finite simplicial $G$-sets. Hence
\[ \cL_{\cC_{G/H}}^{G;H}(\und{R}) \]
can be defined for all $H$-Tambara functors $\und{R}$ whenever $H$ is a
normal subgroup of $G$. 
\end{ex}

\subsection{Finite $\Sigma_n$-simplicial sets with small isotropy for all $n \geq
  2$} \label{sec:permutohedron} 

The $n$th permutohedron is an $(n-1)$-dimensional polytope defined as the
convex hull of the vectors $\sigma(1,2,\ldots, n), \sigma \in \Sigma_n$. 
The symmetric group $\Sigma_n$ acts on the $n$th permutohedron by permuting
the coordinates. Edges connect those vertices whose vectors differ in two
coordinates whose values differ by $1$. 

For $n=2$ this yields the line segment between the points $(1,2)$ and $(2,1)$ in
$\mathbb{R}^2$, for $n=3$ we obtain a hexagon and for $n=4$ one gets a truncated octahedron with four squares and eight hexagons as faces.

In the following we will consider simplicial models of the
$1$-skeleta of the $n$th permutohedron and call them $P_{\Sigma_n}$. For $n=3$
the $1$-skeleton is
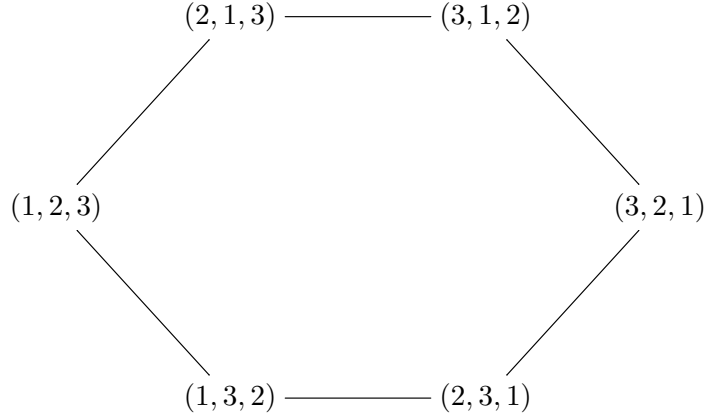
\[ \xymatrix{  & (2,1,3) \ar@{-}[rr]& &(3,1,2) &  \\
    &&&\\
   (1,2,3) \ar@{-}[uur]\ar@{-}[ddr] &&&& (3,2,1) \ar@{-}[uul]\ar@{-}[ddl]
   \\
   &&&\\
    &(1,3,2) \ar@{-}[rr]&& (2,3,1) & 
  } \]

The vertices of the $n$th permutohedron constitute a free orbit
$\Sigma_n/e$ and the stabilizers of the midpoints of the edges are
generated by transpositions, hence Theorem \ref{thm:conjugateisotropy}
applies for all $n$. 
In the example of $n=3$ the barycentric subdivision of the above model
gives rise to  a simplicial model of the $1$-skeleton of the
$3$rd permutohedron, thus the relevant part of $P_{\Sigma_3}$ looks
like this
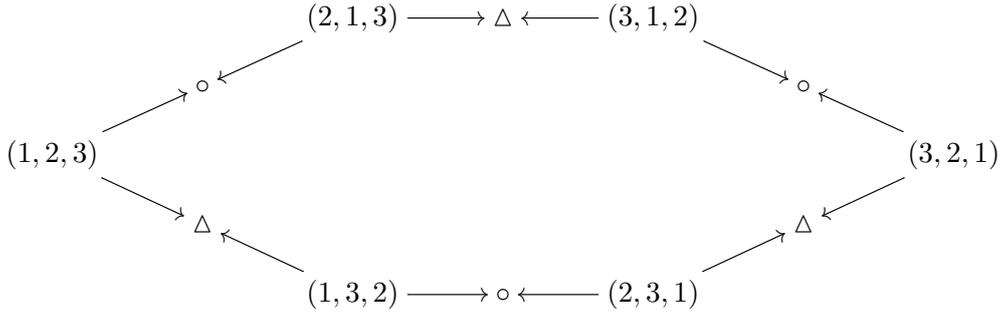
\[ \xymatrix@C=1cm@R=0.4cm{  & & (2,1,3) \ar[r] \ar[dl]& \vartriangle & (3,1,2) \ar[l] \ar[dr] & &  \\
    & \circ &&&& \circ & \\
   (1,2,3) \ar[ur]\ar[dr] &&&&&& (3,2,1) \ar[ul]\ar[dl] \\
 & \vartriangle &&&& \vartriangle & \\
   &&(1,3,2) \ar[r] \ar[ul]& \circ  & \ar[l]  (2,3,1) \ar[ur] & &
  } \]
where the vertices labelled with $(i,j,k)$ for $i,j,k \in
\{1,2,3\}$ constitute a free orbit $\Sigma_3/e$. If we choose as
representatives for the orbits the points in the segment
\[ (1,2.5,2.5) = \vartriangle \longleftarrow (1,2,3) \longrightarrow \circ = (1.5,1.5,3),\]
then the vertices labelled with $\circ$ give rise to a $\Sigma_3/\langle
(1,2)\rangle$-orbit and the $\vartriangle$-vertices assemble into a
$\Sigma_3/\langle (2,3)\rangle$-orbit.  
The half-edges decompose into two free orbits in simplicial degree $1$,
so in total we have 
\[ \Sigma_3/\langle (2,3)\rangle \sqcup \Sigma_3/e \sqcup
  \Sigma_3/e \sqcup \Sigma_3/e \sqcup \Sigma_3/\langle
  (1,2)\rangle \]
in degree one. 
The above structure up to simplicial degree $1$ propagates to
higher simplicial degrees and the resulting Loday 
construction of a $\Sigma_3$-Tambara functor $\und{T}$ in simplicial degree $k$
is 
\begin{align*}
  \cL_{P_{\Sigma_3}}^{\Sigma_3}(\und{T})_k & = 
  (\Sigma_3/\langle (2,3)\rangle \sqcup \Sigma_3/e^{\sqcup k} \sqcup
  \Sigma_3/e \sqcup \Sigma_3/e^{\sqcup k} \sqcup \Sigma_3/\langle
  (1,2)\rangle) \otimes \und{T} \\
& = N_{\langle (2,3)\rangle}^{\Sigma_3}i_{\langle
  (2,3)\rangle}^{\Sigma_3} \und{T} \Box \Big(\bigbox_{i=1}^k
  N_e^{\Sigma_3}i_e^{\Sigma_3} \und{T} \Big) \Box
  N_e^{\Sigma_3}i_e^{\Sigma_3} \und{T} \Box \Big(\bigbox_{i=1}^k
  N_e^{\Sigma_3}i_e^{\Sigma_3} \und{T}\Big) \Box  N_{\langle (1,2)\rangle}^{\Sigma_3}i_{\langle
  (1,2)\rangle}^{\Sigma_3} \und{T}
\end{align*}
and the simplicial structure is the one of two glued two-sided bar
constructions
\[ B(N_{\langle (2,3)\rangle}^{\Sigma_3}i_{\langle
  (2,3)\rangle}^{\Sigma_3} \und{T}, N_e^{\Sigma_3}i_e^{\Sigma_3}
\und{T}, N_e^{\Sigma_3}i_e^{\Sigma_3} \und{T})
\Box_{N_e^{\Sigma_3}i_e^{\Sigma_3} \und{T}}
B(N_e^{\Sigma_3}i_e^{\Sigma_3} \und{T}, N_e^{\Sigma_3}i_e^{\Sigma_3}
\und{T}, N_{\langle (1,2)\rangle}^{\Sigma_3}i_{\langle
  (1,2)\rangle}^{\Sigma_3} \und{T}).  \]

Theorem  \ref{thm:conjugateisotropy}  ensures that there is a
well-defined $\Sigma_n$-Loday construction
$\cL_{P_{\Sigma_n}}^{\Sigma_n;C_2}(\und{R})$ for every
$C_2 := \langle (1,2)\rangle$-Tambara functor $\und{R}$, so  in particular,
we can define $\cL_{P_{\Sigma_3}}^{\Sigma_3;C_2}(\und{R})$ for such
$\und{R}$. We can start with a $\langle (1,2)\rangle$-Tambara 
functor $\und{R}$ and use the inner automorphism of $\Sigma_3$ that is
conjugation by $(1,3)$, $c_{(1,3)}$, and get that
$(c_{(1,3)})^*\und{R}$ is a $\langle (2,3)\rangle$-Tambara functor with
\[ (c_{(1,3)})^*\und{R}(\langle (2,3)\rangle/K) = \und{R}(\langle
  (1,3)(2,3)(1,3)\rangle/(1,3)K(1,3)) = \und{R}(\langle
  (1,2)\rangle/(1,3)K(1,3)). \] 

The $\Sigma_3$-Loday construction of $\und{R}$ and $P_{\Sigma_3}$ is
then 
\[ \cL_{P_{\Sigma_3}}^{\Sigma_3;C_2}(\und{R}) =
  B(N_{\langle (2,3)\rangle}^{\Sigma_3}(c_{(1,3)})^*\und{R}, N_e^{\Sigma_3}i_e^{C_2}
\und{R}, N_e^{\Sigma_3}i_e^{C_2} \und{R})
\Box_{N_e^{\Sigma_3}i_e^{C_2} \und{R}}
B(N_e^{\Sigma_3}i_e^{C_2} \und{R}, N_e^{\Sigma_3}i_e^{C_2}
\und{R}, N_{C_2}^{\Sigma_3}\und{R}).  \]

\subsection{A normal isotropy subgroup and its subgroups}
Assume now that the isotropy subgroups of a finite $G$-simplicial set $X$ are
of the form $H \lhd G$ where $H$ is a normal subgroup of $G$ and some subgroups
$K_i < H$ of which at least one $K_i \neq e$. In this situation all conjugates of the $K_i$s are also contained in
$H$. If we consider an orbit $G/K_i$ and want to merge this with an $H$-Tambara
functor, then defining this as $N_{K_i}^Gi_{K_i}^H\und{R}$ would be problematic
because we cannot control the combination of the isomorphism $\xi$ with the
augmentation maps $\varepsilon \colon N_{K_i}^Hi_{K_i}^H\und{R} \ra \und{R}$. We
opt for a different definition:

\begin{defn} \label{defn:normalandsubgroups}
  If $\und{R}$ is an $H$-Tambara functor and if $X$ is a finite $G$-set with isotropy as above and $X_n = \bigsqcup_{E_H} G/H \sqcup \bigsqcup_i \bigsqcup_{E_{K_i}}G/K_i$ for some finite indexing sets $E_H$ and $E_{K_i}$, then we define
  \[ \cL_X^{G;H}(\und{R})_n := \Big(\bigbox_{E_H} N_H^G\und{R}\Big) \Box \Big(\bigbox_i\bigbox_{E_{K_i}} N_H^GN_{K_i}^Hi_{K_i}^H\und{R}\Big). \]
\end{defn}

\begin{rem}
  In particular, we assign to an orbit $G/K$ with $K < H$ and any $H$-Tambara functor $\und{R}$ the value $N_H^GN_K^Hi_K^H\und{R}$. We interpret this as
  \[ N_H^GN_K^Hi_K^H\und{R} = G \otimes_H (H/K \otimes \und{R}).\]
  Here, $H/K \otimes \und{R}$ is the tensor product of the finite $H$-set $H/K$
  with the $H$-Tambara functor $\und{R}$ using the $H$-commutative monoid structure of
  $H$-Tambara functors and then the outer $G \otimes_H (-)$ takes the resulting
  $H$-Tambara functor and tensors it up to $G$. 

With Definition \ref{defn:normalandsubgroups} we would have produced a slight ambiguity in the case where $H$ is normal in $G$ and the only other isotropy subgroup is trivial. This is why we excluded this case explicitly.  But even in that case the definition would only differ by the application of the isomorphism $\xi \colon N_e^Gi_e^H\und{R} \cong N_H^GN_e^Hi_e^H\und{R}$. 

We restricted to normal subgroups $H \lhd G$ because otherwise it could happen that $K < H$ but $gKg^{-1} \nless H$ for some $g \in G$. Of course $gK_ig^{-1} < gHg^{-1}$ but a fixed given $K$ could then be viewed as a subgroup of different conjugates of $H$ and \emph{a priori} there is no canonical choice.  
\end{rem}

  \begin{prop}
The above definition yields a well-defined $G$-Loday construction.
\end{prop}
\begin{proof}
  We have to prove the functoriality of the above definition. Compositions of projections $G/K_j \ra G/K_i \ra G/H$ have to induce maps
  \[ N_H^GN_{K_j}^Hi_{K_j}^H\und{R} \ra N_H^GN_{K_i}^Hi_{K_i}^H\und{R} \ra N_H^G\und{R} \]
  For a standard projection $\pi_{K_j}^{K_i} \colon G/K_j \ra G/K_i$ for $K_j < K_i$ we use
  the composite
  \[ \xymatrix@1{N_H^GN_{K_j}^H\und{R} \ar[rr]^(0.4){N_H^G(\xi)} &  &
      N_H^GN_{K_j}^HN_{K_i}^{K_j}i_{K_j}^{K_i}i_{K_i}^H\und{R} \ar[rrr]^(0.55){N_H^GN_{K_j}^H(\varepsilon^f)} &&& N_H^GN_{K_i}^Hi_{K_i}^H\und{R}}. \]
  Composing this with the effect of a second standard projection $\pi_{K_i}^H$ yields the composition with
  \[N_H^G(\varepsilon^f) \colon N_H^GN_{K_i}^Hi_{K_i}^H\und{R} \ra N_H^G\und{R}. \]
  The naturality of the norm $N_H^G$ then ensures that this agrees with directly using the standard projection $\pi_{K_j}^H$.

As we showed in Lemma \ref{lem:Weyl} that the Weyl-group action of $W_G(H) = G/H$ on norms $N_H^G$ is natural in the $H$-Tambara variable, we can restrict to standard projections. 

The same argument ensures that composites of equivariant maps 
  $G/{K_j} \ra G/K_i \ra G/K_\ell$ yield compatible maps on Loday constructions. As fold maps induce multiplications, this ensures that the face maps satisfy the simplicial identities. Degeneracies are built out of inclusions $\varnothing \hookrightarrow G/K_i$ or $\varnothing \hookrightarrow G/H$ and the resulting unit maps interact nicely with respect to composition as before. The mixed simplicial identities concern composites of unit maps and multiplications or unit maps and the effect of projections. For the second type of composites note that
  this concerns composites of the form
  \[ \varnothing \hookrightarrow G/K_j \ra G/K_i, \]
  so their effect is a composition
  \[ \xymatrix@1{\und{A} \ar[r]^(0.3)\eta & N_H^GN_{K_j}^Hi_{K_j}^H\und{R}
    \ar[rrr]^{N_H^G(\xi)} &&& N_H^GN_{K_i}^HN_{K_j}^{K_i}i^{K_i}_{K_j}i^H_{K_i} \und{R} \ar[rr]^(0.55){N_H^GN_{K_i}^H(\varepsilon^f)} & & N_H^GN_{K_i}^Hi_{K_i}^H\und{R}. }\]
Here, $\und{A}$ denotes the Burnside Tambara functor for the group $G$. Let's
denote this by $\und{A}_G$ and observe that $N_H^G\und{A}_H \cong \und{A}_G$. The above composite then corresponds to the diagram
\[\xymatrix{N_H^G\und{A}_H \ar@/_3ex/[rrrrrd]_{N_H^G(\eta_{N_{K_i}^Hi_{K_i}^H\und{R}})}
    \ar[rrr]^{N_H^G(\eta_{N_{K_j}^Hi_{K_j}^H\und{R}})} & & & N_H^G(N_{K_j}^Hi_{K_j}^H\und{R}) \ar[rr]^{N_H^G(\xi)} & & N_H^G(N_{K_i}^HN_{K_j}^{K_i}i^{K_i}_{K_j}i^H_{K_i} \und{R})
    \ar[d]^{N_H^G(N_{K_i}^H(\varepsilon^f))}\\
    & & & & & N_H^G(N_{K_i}^Hi_{K_i}^H\und{R}).  } \]
The naturality of the tensor product of finite $H$-sets with $H$-Tambara functors ensures that
\[\eta_{N_{K_i}^Hi_{K_i}^H\und{R}} = N_{K_j}^H(\varepsilon^f) \circ \xi \circ
  \eta_{N_{K_j}^Hi_{K_j}^H\und{R}}) \]
and hence the commutative diagram 
\[ \xymatrix{\varnothing \ar[r] \ar[dr] & G/K_j \ar[d] \\ &  G/K_i} \]
yields a commutative diagram
\[ \xymatrix{
    \und{A}_G \cong N_H^G(\und{A}_H) \ar[dr]\ar[r] & G \otimes_H (H/K_j \otimes \und{R}) \ar[d] \\
    & G \otimes_H (H/K_i \otimes \und{R}). }\]

\end{proof}  
    
\subsection{Two non-trivial isotropy groups (and their conjugates)}  
In the case of Real Hoch\-schild homology we will need the following 
setting: Assume that the occurring isotropy subgroups in a chosen 
orbit decomposition of $X$ are $e$, $H$ and $H'$. Assume that there is an  
  automorphism $\varphi$ of $G$   with $\varphi(H) = H'$.  
  
  Then if $\varphi$ is just conjugation by an element of $g$, we know how to define  $\cL_X^{G;H'}(\und{R})$ for every $H'$-Tambara functor $\und{R}$ by  Theorem
  \ref{thm:conjugateisotropy}.  
  But for Real Hochschild homology, we want a unified treatment of this case and the case that $\varphi$ is a general automorphism where we want to identify orbits as $G/H$ or $G/H'$ and want to make sure that the simplicial structure maps in our $G$-simplicial complex never map between these.  This condition will automatically hold if $H$ and $H'$ are not conjugate since there are no $G$-equivariant maps $G/H\to G/H'$ in this case!  In the case where the subgroups are conjugate, we are making a nontrivial assumption to get a unified construction.  However, if we pick a particular conjugation map $\varphi$ which maps $H$ to $H'$, we can use Lemma \ref{lemma:switchingtoconjugates} to identify the construction we are doing here with the construction in   Theorem
  \ref{thm:conjugateisotropy}.  And if we want to allow general $G$-maps as our simplicial structure maps, we can just use that theorem.

We will show that if all the orbits are identified with $G/e$, $G/H$, and $G/H'$ in a way that  none of the structure maps in $X$ map a summand $G/H$ to $G/H'$ or
  vice versa, we can define
  $\cL_X^{G;H'}(\und{R})$ for every $H'$-Tambara functor $\und{R}$.

Clearly there are examples of subgroups $H$ and $H'$  which are isomorphic via an outer automorphism $\varphi$ of $G$ and in addition there is an equivariant map $G/H \ra G/H'$ (for instance $H=H'$ and so $G/H$ \emph{does} map equivariantly to $G/H'$).  Then the condition of not mapping $G/H$ to $G/H'$ in the simplicial $G$-set is necessary.  The following example illustrates this phenomenon. 
\begin{ex}
Let $G$ be the alternating group on four letters, $A_4$. Then its outer 
automorphism group is isomorphic to $C_2$ and is generated by any conjugation 
by an odd permutation $\sigma \in \Sigma_4$. We consider $\varphi \colon A_4
\ra A_4$ which conjugates by the transposition $(12)$. Then $\varphi$ induces
a non-trivial automorphism on the Klein four group,
   \[  K_4 = \{\id, (12)(34), (13)(24), (14)(23)\}, \]
which fixes the identity and the element $(12)(34)$ and exchanges the other 
two elements. However, this automorphism on $K_4$ is not induced by a
conjugation in $A_4$.
In this case, $K_4 = H = H'$, and $K_4$ is normal in $A_4$.  Hence there are,
of course, non-trivial equivariant self-maps $A_4/K_4 \ra A_4/K_4$, but these
are not induced by $\varphi$. 
\end{ex}

As in the discussion in the beginning of Section \ref{sec:oneisotropy}, we will work with a chosen orbit decomposition of a finite $G$-simplicial set $X$. In the following, we identify each orbit in $X_n$ as $G/e$, $G/H$, or $G/H'$.  
\begin{thm}\label{thm:twosubgroups}
Let $X$ be a finite simplicial $G$-set such that the isotropy subgroups 
in a chosen orbit decomposition of $X$ are of the form $e$, $H$ and
$H'$ where $H$ and $H'$ are proper subgroups of $G$. Assume that there is an
automorphism $\varphi$ of $G$ such that $\varphi(H) = H'$ and that none of the
structure maps in $X$ map an orbit $G/H$ 
to an orbit $G/H'$ or vice versa. 
Then $\cL_X^{G;H'}(\und{R})$ is defined for any $H'$-Tambara functor
$\und{R}$ \end{thm}

\begin{proof}

Every $X_k$ is of the form 
\[   \Big(\bigsqcup_{J_H} G/H\Big) \sqcup \Big(\bigsqcup_{I_e} G/e\Big) \sqcup   \Big(\bigsqcup_{J_{H'}} G/H'\Big)  \]
for finite indexing sets $I_e, J_H, J_{H'}$.    We set
\[ (\cL_X^{G;H'}\und{R})_k :=
 \Big(\bigbox_{J_H}N_{H}^G \varphi^*(\und{R})\Big) \Box \Big(\bigbox_{I_e} N_e^Gi_e^{H'}\und{R}\Big)  \Box \Big(\bigbox_{J_{H'}}  N_{H'}^G\und{R}\Big). \]
Recall that $\varphi^*\und{R}$ is an $H$-Tambara functor as explained in Lemma \ref{lem:varphir}.

Throughout we use the flip-actions on the norms as discussed in Section
\ref{sec:norms}.  
We have to define the simplicial structure maps. Equivariant injective
non-isomorphic maps of finite $G$-sets are inclusions of orbits, hence the
building blocks are of the form $\varnothing \ra G/K$ for $K = e, H'$ or $H$.
Such maps induce the unit maps $\und{A} \ra N_e^Gi_e^{H'}\und{R}$, $\und{A} \ra
N_{H'}^G\und{R}$ and $\und{A} \ra N_H^G\varphi^*\und{R}$ of the
$G$-Tambara functors $N_e^Gi_e^{H'}\und{R}$, $N_{H'}^G(\und{R})$ and
$N_H^G\varphi^*\und{R}$. 

Isomorphisms are built out of permutations of summands and
equivariant self-maps. Permutation induce permutations of
$\Box$-factors. Equivariant self-maps of $G/e$, $G/H$ or $G/H'$ come from the respective
Weyl group actions, whose effects on our Tambara constructions were explained
in Lemma \ref{lem:Weyl}.  By our assumption,  equivariant maps between $G/H$ and $G/H'$ 
will not occur in the simplicial structure maps.. 

Fold maps induce multiplication. A non-isomorphic surjection on orbits
has to involve projections $G/e \ra G/H$ or $G/e \ra G/H'$.

For $G/e \ra G/H'$ we use the map  
\[ \xymatrix@1{N_e^Gi_e^{H'}\und{R} \ar[r]^(0.45)\xi &
    N_{H'}^GN_e^{H'}i_e^{H'}\und{R} \ar[rr]^(0.6){N_{H'}^G(\varepsilon^f)}
    & & N_{H'}^G\und{R}}\]
from \eqref{eq:counitgh'} and for $G/e \ra G/H$ we use the map from \eqref{eq:counitgh} 
\[\xymatrix@1{N_e^Gi_e^{H'}\und{R} \ar[r]^(0.45)\xi & N_{H}^GN_e^{H}i_e^{H'}\und{R} \ar[rr]^(0.55){N_{H}^G(\varepsilon^f_\varphi)} & & N_{H}^G(\varphi^*\und{R}).}  \]

The face maps in the simplicial structure induce fold maps combined
with surjections on orbits. 

Degeneracies induce insertions of the unit map. The simplicial
identities can then be checked by a direct  argument as in the proof
of Theorem \ref{thm:conjugateisotropy}.

\end{proof}
\begin{ex}
  For $m \geq 1$  let $D_{2m}$ be the dihedral group  with $2m$-elements. We will use the presentations
  \begin{align*}
    D_{2m} = & \langle r,s \mid s^2=e=r^m, sr^{m-1}=rs\rangle \\
    = & \langle r,s \mid s^2=e=r^m, srs=r^{-1}\rangle.
\end{align*}
We consider the subgroups of order two $D_2 = \langle s \mid s^2=e \rangle$ and
$D'_2 = \langle rs \mid (rs)^2=e\rangle$. We will consider $\varphi\colon D_{2m} \ra D_{2m}$ defined by $\varphi(rs) = s$ and $\varphi(r)=r$. Then $\varphi(D'_2) = D_2$. 
If $m = 2n+1$, then $D_2$ and $D'_2$ are actually conjugate in $D_{4n+2}$ because
in that case 
\[ r^{-n}sr^n = r^{-2n}s = rs. \]
Thus for odd $m$ the automorphism $\varphi$ is inner. The subgroups $D_2$ and $D'_2$ are \emph{not} conjugate in $D_{2m}$ if $m$ is even and then
$\varphi$ represents an outer automorphism. In \cite{akgh}
 the authors still use the notation $c_\zeta$ for the isomorphism
 between $D_2$ and $D'_2$ in all cases. 
\end{ex}

\section{Real Hochschild homology} \label{sec:RealHH}

Real Hochschild homology is defined in \cite[Definition 4.7]{akgh} for discrete $E_\sigma$-rings $\und{R}$. Explicitly, such a discrete $E_\sigma$-ring can be
described as follows \cite[Lemma 4.8]{akgh}
\begin{enumerate}
\item 
  It is a $D_2$-Mackey functor $\und{R}$ together with an associative ring structure on $\und{R}(D_2/e)$ for which the $D_2$-action acts as an anti-involution.
\item
  There is an $\tilde{N}_e^{D_2}i_e^*\und{R}$-module structure on $\und{R}$ that coincides with the standard $\und{R}(D_2/e) \otimes \und{R}(D_2/e)^{op}$-module structure on $\und{R}(D_2/e)$.
\item There is an element $1 \in \und{R}(D_2/D_2)$ such that its restriction
  $\res(1)$ is the unit $1 \in \und{R}(D_2/e)$. 
\end{enumerate}
Here, $\tilde{N}$ denotes the Mackey norm functor.

Of course, $D_2$-Tambara functors satisfy these properties, but also for
instance Mackey functors of the form $\und{R}^\fix$ for an associative ring
$R$ with anti-involution.

Real $D_{2m}$-Hochschild homology of an $E_\sigma$-ring $\und{R}$ is then
defined as the homotopy groups of a two-sided bar construction: 
\[ \und{HR}_*^{D_{2m}}(\und{R}) := \pi_* B(N_{D_2}^{D_{2m}}\und{R}, N_e^{D_{2m}}i_e^*\und{R}, N_{D'_2}^{D_{2m}}c_\zeta^*(\und{R})). \]

In order to model Real Hochschild homology via equivariant Loday
constructions, we first have to find suitable simplicial sets with a
$D_{2m}$-action for all $m \geq 1$:

\begin{prop}
  For all $m \geq 1$ the dihedral group $D_{2m}$ with $2m$ elements acts on the 
$1$-skeleton of a regular $2m$-gon. 
\end{prop}
We denote this $1$-skeleton by $P_{2m}$.

\begin{proof}
For $m \geq 1$ and $D_{2m}$, the non-degenerate simplices of $P_{2m}$ are two orbits in degree zero
\[ D_{2m}\cdot\{x_0\} = \{x_0, x_1,\ldots,x_{m-1}\} \cong D_{2m}/\{e,s\}
  \text{ and } D_{2m}\cdot\{x'_0\}=\{x'_0, x'_1 , \ldots, x'_{m-1}\} \cong 
  D_{2m}/\{e,rs\}\]
and a free orbit in degree one $D_{2m}\cdot\{y_e\}=\{y_g \mid g \in D_{2m}\}$. The face maps are determined by 
\begin{align*}
  d_0: D_{2m}/e \{y_e\} & \to D_{2m}/\{e, rs\} \{x'_0\},\\
  y_{r^k} & \mapsto x'_k ,\\
  y_{r^ks} & \mapsto x'_{k-1};\\
  d_1: D_{2m}/e \{y_e\} & \to D_{2m}/\{e, s\} \{x_0\},\\
  y_{r^k} & \mapsto x_k ,\\
  y_{r^ks} & \mapsto x_{k}.
\end{align*}
\end{proof}

We draw the first low dimensional cases: For $m=1$ we obtain the $1$-skeleton of a $2$-gon and this can be identified with  $S^\sigma$, the simplicial model of
the $1$-point compactification of the real sign-representation: 
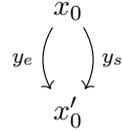
\[\xymatrix{x_0 \ar@/_2ex/[d]_{y_e} \ar@/^2ex/[d]^{y_s}\\
  x'_0 }\]
Here, $x_0$ and $x'_0$ have the trivial action and correspond to the
$D_2$-set $D_2/D_2$ whereas the $1$-simplices form a free $D_2$-orbit
$D_2/e$.

\medskip 
For $m=2$ the group $D_4$ acts on
$
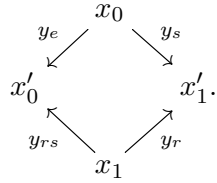
\xymatrix@C=0.5cm@R=0.5cm{  & x_0\ar[dr]^{y_s} \ar[dl]_{y_e} & \\
    x'_0 & & x'_1. \\
    & x_1 \ar[ur]_{y_r} \ar[ul]^ {y_{rs}}&}$
  
The reflection $s$ along the vertical axis fixes $x_0$ and $ x_1$ but
interchanges $x'_0$ and $x'_1$ whereas the rotation by $180$-degrees
has the orbits $\{x_0, x_1\} \cong D_4/\langle s\rangle$ and
$\{x'_0, x'_1\} \cong D_4/\langle rs\rangle$. Note that $rs=sr$ in this
case. The $1$-simplices form a free $D_4$-orbit. 

\medskip
For $m = 3$ the symmetric group on three elements aka $D_6$ acts on the
$1$-skeleton of the regular hexagon
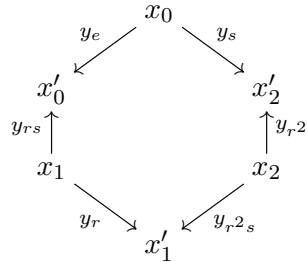
\[ \xymatrix@C=0.8cm@R=0.5cm{
  & x_0 \ar[dr]^{y_s} \ar[dl]_{y_e}& \\
  x'_0 & & x'_2\\
 x_1  \ar[u]^{y_{rs}}\ar[dr]_{y_r} & & x_2 \ar[u]_{y_{r^2}} \ar[dl]^{y_{r^2s}}\\
  & x'_1 &   }\]

The orbits of the rotation by $120$-degrees, $r$, on the zero-cells are
$\{x_0,x_1,x_2\}$ and $\{x'_0, x'_1,x'_2\}$, whereas the reflection
at the vertical axis fixes $x_0$ and $x'_1$, but has the non-trivial
orbits $\{x'_0,x'_2\}$ and $\{x_1,x_2\}$. 
Thus the zero-cells form two $D_6$-orbits, $D_6x_0 = \{x_0,x_1,x_2\}$
and $D_6x'_0 = \{x'_0,x'_1,x'_2\}$. Again, the $1$-simplices form a
free $D_6$-orbit. 

\medskip
For $D_8$ the situation becomes more interesting. We consider the
$1$-skeleton of the regular octagon:

\newpage
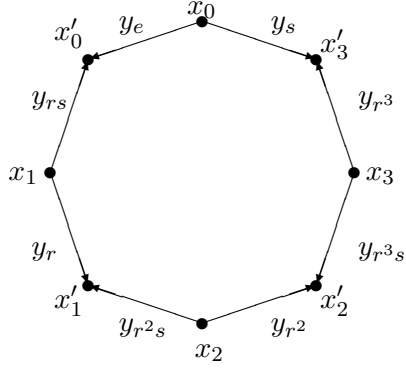
\mbox{}\\
\vspace{1cm}

\hspace{4cm}
\begin{picture}(10,5)
\setlength{\unitlength}{0.5cm}
\put(0,0){$\bullet$}
\put(-0.9,0){$x_1$}
\put(1,3){$\bullet$}
\put(0.3,3.7){$x'_0$}
\put(4,4){$\bullet$}
\put(3.8,4.4){$x_0$}

\put(2,4){$y_e$}
\put(6,4){$y_s$}
\put(2,-4){$y_{r^2s}$}
\put(6,-4){$y_{r^2}$}
\put(-0.3,2){$y_{rs}$}
\put(8.3,2){$y_{r^3}$}
\put(-0.3,-2){$y_r$}
\put(8.3,-2){$y_{r^3s}$}

\put(7,3){$\bullet$}
\put(7.3,3.4){$x'_3$}

\put(8,0){$\bullet$}
\put(8.5,0){$x_3$}
\put(7,-3){$\bullet$}
\put(7.3,-3.4){$x'_2$}

\put(4,-4){$\bullet$}
\put(4,-4.7){$x_2$}

\put(1,-3){$\bullet$}
\put(0.3,-3.4){$x'_1$}
\put(0.2,0.2){\vector(1,3){1}}
\put(0.2,0.2){\vector(1,-3){1}}
\put(8.2,0.2){\vector(-1,3){1}}
\put(8.2,0.2){\vector(-1,-3){1}}
\put(4.2,4.2){\vector(3,-1){3}}
\put(4.2,4.2){\vector(-3,-1){3}}
\put(4.2,-3.8){\vector(3,1){3}}
\put(4.2,-3.8){\vector(-3,1){3}}
\end{picture}

\vspace{3cm}

Here the orbit of $x_0$ is $\{x_0,x_1,x_2,x_3\} \cong D_8/\langle
s\rangle$ and the orbit of $x'_0$ is $\{x'_0,x'_1,x'_2,x'_3\} \cong D_8/\langle
sr^3=rs\rangle$.
Note that for instance the stabilizers of $x_1$ and $x_3$ are
  $\langle r^2s = sr^2\rangle$ so there are choices
  involved. Similarly for the other orbit we get $\langle r^3s =
  sr\rangle$ as stabilizers of $x_1'$ and $x_3'$.


\bigskip

\begin{rem}
  There is a different beautiful family of graphs on which the groups $D_{2m}$
  act. We thank Imma G\'alvez for drawing our attention to them. Ruedi Suter
  described a subgraph of the Hasse graph of the Young lattice for all $n \geq 3$ consisting of
  those Young diagrams whose hull is contained in the staircase diagram of
  size $n$. For $n=3$ one obtains the graph 
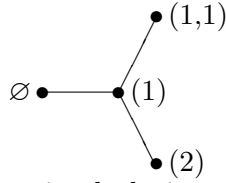
\mbox{}\\
\vspace{0.5cm}

\hspace{4cm}
\begin{picture}(5,5)
\setlength{\unitlength}{0.5cm}
\put(1.8,-0.2){$\bullet$}
\put(2.3,-0.2){$(1)$}
\put(-0.2,-0.2){$\bullet$}
\put(-0.9,-0.2){$\varnothing$}
\put(2.8,1.8){$\bullet$}
\put(3.3,1.8){(1,1)}
\put(2.8,-2.1){$\bullet$}
\put(3.3,-2.1){(2)}
\put(0,0){\line(1,0){2}}
\put(2,0){\line(1,2){1}}
\put(2,0){\line(1,-2){1}}
\end{picture}

\vspace{1cm}

\noindent
The rotation in $D_6$ acts via clockwise rotation by $120$ degrees. The reflection action is by reflection at the $x$-axis, exchanging the partitions $(1,1)$ and $(2)$. For $n=4$ we obtain the graph 

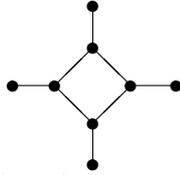
\mbox{}\\
\vspace{0.5cm}

\hspace{4cm}
\begin{picture}(5,5)
\setlength{\unitlength}{0.5cm}

\put(0.7,-0.2){$\bullet$}
\put(0.9,0){\line(1,0){1}}
\put(1.8,-0.2){$\bullet$}
\put(2,0){\line(1,1){1}}
\put(2,0){\line(1,-1){1}}

\put(2.8,0.8){$\bullet$}
\put(2.8,-1.2){$\bullet$}
\put(3,1){\line(0,1){1.1}}
\put(3,-1.1){\line(0,-1){1.1}}

\put(2.8,1.9){$\bullet$}
\put(2.8,-2.3){$\bullet$}

\put(4,0){\line(1,0){1.1}}

\put(4,0){\line(-1,1){1}}
\put(4,0){\line(-1,-1){1}}

\put(3.8,-0.2){$\bullet$}
\put(5,-0.2){$\bullet$}
\end{picture}

\vspace{1cm}

\noindent
The inner vertices correspond to the partitions $(1)$, $(1,1)$, $(2,1)$ and
$(2)$ and the outer vertices to $\varnothing$, $(1,1,1)$, $(2,2)$, and $(3)$. The rotation in $D_8$ acts by clockwise rotation by $90$ degrees and the reflection sends a partition $\lambda$ to its transpose $\lambda^t$.  
Its associated simplicial $D_8$-set is an octagon with $4$ spikes and is equivalent to $P_8$ because the outer edges can be contracted. 

For $n$  odd  we always get a trivial orbit in the center corresponding to a staircase diagram of half the size. For $n$ even the isotropy subgroups grow with $n$. For $n =6$ we already have a stabilizer of order $6$ for the $D_{12}$-action. You can find the explicit description of the rotation action and the  pictures for $n$ up to $8$ in \cite{suter}.
\end{rem}

As $P_{2m}$ is a simplical $D_{2m}$-set, we can define the ordinary
$D_{2m}$-equivariant Loday construction of any $D_{2m}$-Tambara functor
$\und{T}$ with respect to
$P_{2m}$: 
\[
\cL_{P_{2m}}^{D_{2m}}(\und{T}) 
\]
\begin{lem}
  For any $D_{2m}$-Tambara functor $\und{T}$ there is an isomorphism of
  simplicial $D_{2m}$-Tambara functors 
  \[ \cL^{D_{2m}}_{P_{2m}}(\und{T}) 
\cong B(N^{D_{2m}}_{D_2} i^{D_{2m}}_{D_2}\und{T}, N^{D_{2m}}_e i^{D_{2m}}_e \und{T}, N^{D_{2m}}_{D'_2} i^{D_{2m}}_{D'_2}\und{T}). \] 
\end{lem}
\begin{proof}
  The orbit structure of $P_{2m}$ ensures that we obtain an isomorphism for every fixed simplicial degree $n$:
\[ \cL^{D_{2m}}_{P_{2m}}(\und{T})_n 
\cong B_n(N^{D_{2m}}_{D_2} i^{D_{2m}}_{D_2}\und{T}, N^{D_{2m}}_e i^{D_{2m}}_e \und{T}, N^{D_{2m}}_{D'_2} i^{D_{2m}}_{D'_2}\und{T}). \]

The face maps in $\cL^{D_{2m}}_{P_{2m}}(\und{T})$ and in the two-sided bar construction are induced by norms of the multiplicative counit maps
$N^{D_2}_ei^{D_{2}}_ei_{D_2}^{D_{2m}} \und{T} \to i^{D_{2m}}_{D_2}\und{T}$ and $N^{D'_2}_ei^{D'_{2}}_ei_{D'_2}^{D_{2m}} \und{T} \to
i^{D_{2m}}_{D'_2}\und{T}$ and by fold maps. The degeneracies are given by unit insertions in the same spots in both simplicial Tambara functors. 
\end{proof}

As a consequence of Theorem \ref{thm:twosubgroups} we obtain the following result: 
\begin{thm} \label{thm:RealHH}
  Assume that $\und{R}$ is a $D_2$-Tambara functor and let $\varphi \colon D_{2m} \cong D_{2m}$ be the automorphism defined by $\varphi(rs)=s$ and $\varphi(r)=r$. Then
  \[ \mathcal{L}_{P_{2m}}^{D_{2m};D_2}(\und{R}) \cong B(N^{D_{2m}}_{D_2}
    \und{R}, N^{D_{2m}}_e i_e \und{R}, N^{D_{2m}}_{D'_2} \varphi^*(\und{R}))
  \]
  and this is isomorphic to the Real $D_{2m}$-Hochschild homology
  $\und{HR}_*^{D_{2m}}(\und{R})$ of $\und{R}$ viewed as a discrete
  $E_\sigma$-ring with values in graded $D_{2m}$-Mackey functors. 
\end{thm}

Usually one needs commutativity, hence Tambara functors, in order to define
equivariant Loday constructions. In the cases of the $1$-skeleta of the regular
$2m$-gons with the $D_{2m}$-action we just need a discrete $E_\sigma$-ring
$\und{R}$. 


\begin{prop} \label{prop:esigma}
Let $\und{R}$ be a discrete $E_\sigma$-ring. Then 
$\mathcal{L}_{P_{2m}}^{D_{2m};D_2}(\und{R})$ is defined for all $m \geq 1$ and
coincides with the Real $D_{2m}$-Hochschild homology of $\und{R}$. 
\end{prop}
\begin{proof}
  For $m=1$ we already showed in \cite[\S 9]{lr} that
  $\mathcal{L}_{P_{2}}^{D_{2}}(\und{R}) = \mathcal{L}_{S^\sigma}^{D_{2}}(\und{R})$
  can be defined for fixed point functors associated with associative rings
  where $D_2$ acts by an anti-involution. The general case of discrete
  $E_\sigma$-rings  is similar:
  We define $D_2 \otimes_e \und{R}$ to be
  $\tilde{N}_e^{D_2}i_e^{D_2}\und{R}$ and $D_2 \otimes_{D_2} \und{R} := \und{R}$.
  As  $S^\sigma$ consists of two copies of $\Delta(-,[1])$ that are glued
  together at the endpoints, the linear ordering of the simplices of
  $\Delta(-,[1])$ ensures that the associativity of
  $\tilde{N}_e^{D_2}i_e^{D_2}\und{R}$ and the
  $\tilde{N}_e^{D_2}i_e^{D_2}\und{R}$-bimodule structure on $\und{R}$ suffice to
  define $\cL_{S^\sigma}^{D_2}(\und{R})$: In simplicial level $n$ we obtain
  \[ (S^\sigma)_n \otimes \und{R} = (D_2/D_2 \sqcup D_2/e^{\sqcup n} \sqcup D_2/D_2) \otimes \und{R} = \und{R} \Box \Big(\bigbox_{i=1}^n
    \tilde{N}_e^{D_2}i_e^{D_2}\und{R}\Big) \Box \und{R}. \]
  The degeneracies just insert units and the face maps induce multiplication in
  $\tilde{N}_e^{D_2}i_e^{D_2}\und{R}$ or the
  $\tilde{N}_e^{D_2}i_e^{D_2}\und{R}$-bimodule action on $\und{R}$.

  For $m > 1$ we observe that the simplicial structure of $P_{2m}$ is determined
  by the $0$- and $1$-simplices of $P_{2m}$. 
In degree zero we have two orbits 
\[ D_{2m}\cdot\{x_0\} = \{x_0, x_1,\ldots,x_{m-1}\} \cong D_{2m}/\{e,s\}
  \text{ and } D_{2m}\cdot\{x'_0\}=\{x'_0, x'_1 , \ldots, x'_{m-1}\} \cong 
  D_{2m}/\{e,rs\}\]
and the non-degenerate $1$-simplices constitute a free orbit $D_{2m}\cdot\{y_e\}=\{y_g \mid g \in D_{2m}\}$. As the face maps are determined by 
\begin{align*}
  d_0: D_{2m}/e \{y_e\} & \to D_{2m}/\{e, rs\} \{x'_0\},\\
  y_{r^k} & \mapsto x'_k ,\\
  y_{r^ks} & \mapsto x'_{k-1};\\
  d_1: D_{2m}/e \{y_e\} & \to D_{2m}/\{e, s\} \{x_0\},\\
  y_{r^k} & \mapsto x_k ,\\
  y_{r^ks} & \mapsto x_{k}, 
\end{align*}
  we only need the $\tilde{N}_e^{D_2}i_e^{D_2}\und{R}$-bimodule structure on $\und{R}$ to define these. Again, in higher degrees there is a linear ordering on the edges that constitute the free orbits, because each edge corresponds to a copy of $\Delta(-,[1])$. The face maps in higher degrees induce the multiplication on $\tilde{N}_e^{D_2}i_e^{D_2}\und{R}$ and the $\tilde{N}_e^{D_2}i_e^{D_2}\und{R}$-bimodule structure on $\und{R}$. 
\end{proof}

In the setting of equivariant stable homotopy theory we can define equivariant
Loday constructions relative to isotropy subgroups as for Tambara functors
using smash products of spectra instead of box products, using the Hill-Hopkins-Ravenel norms from \cite{hhr} instead of norms of Tambara or Mackey functors and using the change-of-groups
pullback as in Lemma \ref{lem:pi0}. Thus if $B$ is a genuine commutative $H'$ ring spectrum, and $X$ is a finite simplicial $G$-set such that in a chosen orbit
decomposition only the orbit types $G/e$, $G/H$ and $G/H'$ occur with $H,H'$ as
above, then we can set $G \otimes_e B := N_e^Gi_e^HB$, $G \otimes_{H'} B := N_{H'}^GB$ and $G \otimes_H B := N_H^G\varphi^*B$ where $\varphi$ is an automorphism of $G$ with $\varphi(H) = H'$. There is then a spectral analogue of Theorem \ref{thm:twosubgroups} and as a corollary we obtain an identification of all
$D_{2m}$-restrictions of the $O(2)$-spectrum of
Real topological Hochschild homology. There is a notion of an $E_\sigma$-ring spectrum (see \cite[Corollary 2.2]{akgh}). Any genuine ring spectrum where $D_2$ acts by anti-involution is an example \cite[Example 2.3]{akgh}. 

\begin{prop} \label{prop:thr}
If $A$ is an $E_\sigma$-ring spectrum, then  
\[ i^{O(2)}_{D_{2m}}\thr(A) \simeq \cL_{P_{2m}}^{D_{2m};D_2}(A).  \]
Moreover, if $A$ is connective  this yields an isomorphism of simplicial $D_{2m}$-Mackey functors on $\und{\pi}_0$ 
\[ \und{\pi}_0^{D_{2m}}(i^{O(2)}_{D_{2m}}\thr(A)) \cong \und{\pi}_0^{D_{2m}}\cL_{P_{2m}}^{D_{2m};D_2}(A) \cong \cL_{P_{2m}}^{D_{2m};D_2}(\und{\pi}_0^{D_2} A). \]
\end{prop}
\begin{proof}
  We prove the claim about the zeroth homotopy group. In \cite[Proposition 6.2]{lrz-gloday} we showed that there is an isomorphism of simplicial $G$-Tambara functors
  \[\und{\pi}_0^G(\cL_X^G(R)) \cong \cL_X^G(\und{\pi}_0^G(R))\]
  if $X$ is a finite $G$-simplicial set and $R$ is a connective genuine commutative $G$ ring spectrum and $\und{\pi}_0^G(R)$ is the associated $G$-Tambara functor. The proof was based on the facts that for a connective $R$ also the Hill-Hopkins-Ravenel norms of $R$ are connective and $\und{\pi}_0^G(N_H^Gi_H^GR) \cong N_H^Gi_H^G(\und{\pi}_0^G(R))$ and that $\und{\pi}_0^G$ turns smash products of spectra into box products.

In our case we have $(P_{2m})_n = D_{2m}/D_2 \sqcup \Big(\bigsqcup_{i=1}^n D_{2m}/e\Big)  \sqcup D_{2m}/D'_2$ and hence 
\[ \cL_{P_{2m}}^{D_{2m};D_2}(A)_n = N_{D_2}^{D_{2m}}(A) \wedge \Big(\bigwedge_{i=1}^n N_e^{D_{2m}}i_e^{D_2} A\Big) \wedge N_{D'_2}^{D_{2m}} \varphi^*A. \]

Applying $\und{\pi}_0$ yields with Lemma \ref{lem:pi0}
\begin{align*}
  \und{\pi}_0^{D_{2m}}\Big(\cL_{P_{2m}}^{D_{2m};D_2}(A)_n) & = \und{\pi}_0^{D_{2m}}\Big(N_{D_2}^{D_{2m}}(A) \wedge \Big(\bigwedge_{i=1}^n N_e^{D_{2m}}i_e^{D_2} A\Big) \wedge N_{D'_2}^{D_{2m}} \varphi^*A\Big) \\
  & \cong N_{D_2}^{D_{2m}}(\und{\pi}_0^{D_2}(A)) \Box \Big(\bigbox_{i=1}^n N_e^{D_{2m}}i_e^{D_2} \und{\pi}_0^{D_2}(A)\Big) \Box N_{D'_2}^{D_{2m}} \varphi^*\und{\pi}_0^{D_2}(A). \end{align*}
We therefore get an isomorphism in every simplicial level. As $\und{\pi}_0$ is
a strong symmetric monoidal functor, these levelwise isomorphisms combine to
give an isomorphism of simplicial Mackey functors. 
\end{proof}

\begin{rem}
  We hope that the above result is useful for establishing extra structure
  on Chloe Lewis' B\"okstedt type spectral sequence, similar to the Hopf algebra structure for the ordinary B\"okstedt spectral sequence established by Angeltveit and Rognes \cite{arognes}. Lewis showed for instance that the $D_2$-restriction of $\thr(A)$ carries a Hopf algebroid structure in the
  $D_2$-equivariant stable homotopy category if $A$ is commutative \cite[Theorem 5.11]{lewis}. Using the model $i_{D_2}^{O(2)}\thr(A) \simeq \cL_{S^\sigma}^{D_2}(A)$ this is relatively easy to see using geometrically defined maps such as pinch and fold maps on (subdivisions of) $S^\sigma$. Of course the situation for the $D_{2m}$-restrictions for $m > 1$ is more involved, but we plan to investigate extra structures on these in future work. 
  
\end{rem}

\begin{rem}
As $O(2) \cong S^1 \cup S^1\rho$ with $\rho$ being the reflection
across the $x$-axis (see e.g. \cite[\S 5]{akmp}), there is another
naturally occurring Loday construction related to $O(2)$: The standard
simplicial model of $O(2)$ has $D_{2n+2}$ as $n$-simplices and as
$D_{2n+2}$ is a semi-direct product of $C_{n+1}$ and $D_2$ we can
model $S^1 \cup S^1\rho$ in degree $n$ by two copies of the
$1$-skeleton of a regular $n+1$-gon, $P'_{n+1}$ where $D_2$ interchanges the
$2$-copies. Note that the orientation of the edges in $P'_{n+1}$ is dictated by
the rotation action of $C_{n+1}$ on $P'_{n+1}$. We denote the corresponding
simplicial $D_{2m+2}$-set also  
by $P'_{n+1} \sqcup P'_{n+1}$. Note that this disjoint union of
simplicial sets is equivariantly not disjoint. 

However, the corresponding Loday construction $\cL_{P'_{n+1} \sqcup
  P'_{n+1}}^{D_{2n+2}}(\und{R})$ is very different from the one of
Theorem \ref{thm:RealHH}: In every simplicial degree $k$ we get that
$(P'_{n+1} \sqcup P'_{n+1})_k$ is a disjoint union of free orbits. 
In particular, this Loday construction is \emph{not} related to Real
Hochschild homology. 
\end{rem}

\end{document}